\documentclass[11pt]{amsart}
\usepackage{amsmath,amssymb,amsthm,mathtools}
\usepackage[margin=1.1in]{geometry}
\usepackage{graphicx}
\usepackage{appendix}
\usepackage{hyperref}
\hypersetup{colorlinks=true,
            linkcolor=blue,
            anchorcolor=blue,
            citecolor=blue,
            pdftitle={Non-homogeneous curvature flows in a hemisphere},
            pdfauthor={Hongyi Sheng, Weimin Sheng, Jiazhuo Yang},
            pdfkeywords={non-homogeneous curvature flow, convexity, hemisphere, elementary symmetric functions}}

\newtheorem{theorem}{Theorem}[section]
\newtheorem{lemma}[theorem]{Lemma}

\newtheorem{corollary}[theorem]{Corollary}
\theoremstyle{definition}

\theoremstyle{remark}
\newtheorem{remark}[theorem]{Remark}

\numberwithin{equation}{section}

\newcommand{\Snn}{\mathbb{S}^{n+1}}
\newcommand{\Snp}{\mathbb{S}^{n+1}_{+}}

\def\p{\partial}
\def\n{\nabla}

\let\ringaccent\r
\def\r{\rangle}

\title[Non-homogeneous curvature flows in a hemisphere]
{Non-homogeneous curvature flows in a hemisphere}

\author{Hongyi Sheng}
\address{Institute for Theoretical Sciences, Westlake Institute for Advanced Study, Westlake University,
Hangzhou 310030, China}
\email{shenghongyi@westlake.edu.cn}

\author{Weimin Sheng}
\address{School of Mathematical Sciences, Zhejiang University, Hangzhou 310058, China}
\email{weimins@zju.edu.cn}

\author{Jiazhuo Yang}
\address{School of Mathematical Sciences, Zhejiang University, Hangzhou 310058, China}
\email{yangjiazhuo@zju.edu.cn}

\date{}


\begin{document}

\begin{abstract}
Let $\Snp$ be the open hemisphere of the unit sphere $\Snn$ centred at
$o$.  We study the non-homogeneous curvature flow
$\partial_tX=-f(\rho)\sigma_k^\alpha\nu$ of smooth, closed, strictly
convex hypersurfaces enclosing $o$, where $\rho$ is the geodesic
distance to $o$.  We consider both the supercritical regime
$\beta>1+k\alpha$ and the critical regime $\beta=1+k\alpha$, where
$\beta$ is the growth order of the profile at the origin,
$f(r)\sim r^\beta$ as $r\downarrow0$.  Under the structural
condition that $f^{1/(1+k\alpha)}$ is convex, we prove long-time
existence and preservation of strict convexity (for $n\ge 2$ in the critical case).  The normalized radial
function converges smoothly and exponentially to a constant: to $1$ in
the supercritical case and to a data-dependent constant $R_\infty>0$
in the critical case.  Thus the normalized radial graphs become round,
while the original hypersurfaces contract to $o$. 
\end{abstract}

\maketitle
\tableofcontents

\section{Introduction}

Curvature flows of closed convex hypersurfaces have been studied extensively
over the past several decades. In Euclidean space, consider the contracting flow
\[
\partial_tX=-F(\kappa)\nu,
\]
where $F$ is a smooth symmetric function of the principal curvatures, strictly
increasing in each argument. For $F=H$, Huisken \cite{Hui84} proved that every smooth,
closed, strictly convex hypersurface contracts to a round point in finite time.
Chow \cite{Chow85, Chow87} established the corresponding results for $F=\sigma_n^{1/n}$ and
$F=\sigma_2^{1/2}$, and Andrews \cite{And94, And07} extended them to general degree-one
homogeneous speeds which are convex, or concave and either vanishing on the
boundary of the positive cone or inverse-concave. This completes the classical
degree-one theory, in which convex solutions always contract to round points.

For speeds not homogeneous of degree one, the limiting shape depends
essentially on the exponent. The model problem is the flow by powers of the
Gauss curvature, $\partial_tX=-K^\alpha\nu$: Tso \cite{Tso85} proved contraction to a
point for $\alpha=1$, and Andrews \cite{And99} proved convergence to a round point for
$n=2$, resolving Firey's conjecture \cite{Firey74}. Through the combined works of
Andrews--Chen \cite{AC12}, Guan--Ni \cite{GN17}, Andrews--Guan--Ni \cite{AGN16}, and
Brendle--Choi--Daskalopoulos \cite{BCD17}, every smooth, closed, strictly convex
solution is now known to contract to a round point throughout the range
$\alpha>1/(n+2)$, while at the critical exponent $\alpha=1/(n+2)$,
corresponding to the affine normal flow, the normalized limit is an ellipsoid
\cite{BCD17}. For the more general flows $\partial_tX=-\sigma_k^\alpha\nu$,
Li--Wang--Wu \cite{LWW21} proved that, for closed, strictly convex, axially symmetric
hypersurfaces in Euclidean space and in the sphere, the properly rescaled flow
converges exponentially to a sphere whenever $\alpha\in[1/k,c(n,k)]$ for some
$c(n,k)>1/k$.

Curvature flows have also been studied in non-Euclidean space forms. In
hyperbolic space, Andrews--Chen \cite{AC} proved finite-time contraction and
asymptotic roundness for compact surfaces with positive intrinsic scalar
curvature under several degree-one speeds, and proved the analogous
result for mean
curvature flow in higher dimensions under positive intrinsic Ricci curvature.
In the sphere, Gerhardt \cite{Ger15} used polar duality to couple the contracting flow
with speed $F$ to the expanding flow with dual speed $\tilde F^{-1}$: the
contracting hypersurfaces shrink to a point, the expanding ones converge to
the equator of the opposite hemisphere, and after rescaling both flows become
exponentially round. Chen--Huang \cite{CH} studied the flow by $K^\alpha$ in the
sphere and in hyperbolic space, proving contraction to a point for every
$\alpha>0$ and smooth convergence of the rescaled hypersurfaces to a geodesic
sphere when $\alpha>1/(n+2)$. Spherical polar duality also underlies the work
of Guang--Li--Wang \cite{GLW24}, who combined a variational min--max construction with
a Gauss curvature flow to solve the Minkowski problem in the sphere.

A parallel line of work allows the speed to depend explicitly on position. In
Euclidean space, Li--Sheng--Wang introduced the anisotropic Gauss curvature
flow with speed $f(\nu)r^\beta K$ as a parabolic approach to the Aleksandrov
and dual Minkowski problems \cite{LSWJEMS}, and subsequently treated the fully nonlinear
speed $r^\beta\sigma_k$ \cite{LSW20}. Li--Xu--Zhang \cite{LXZ} extended the convergence
theory to the power-type speed $r^{\alpha/\beta}\sigma_k^{1/\beta}$ on
star-shaped $k$-convex hypersurfaces, and Sheng--Yang \cite{SY} replaced the power
of $r$ by a general function and studied the non-homogeneous speed
$f(r)\sigma_k^\alpha$. In hyperbolic space, Hong \cite{Hong21} proved exponential
convergence of the normalized flow with speed
$(\sinh\rho)^{\alpha/\beta}\sigma_k^{1/\beta}$ to a geodesic sphere, and the
authors' companion paper \cite{hyp} treats the hyperbolic analogue of the flow
studied here.

Motivated by these developments, we study the spherical counterpart of the
radial flows above. Let $\mathbb{S}^{n+1}$ be the $(n+1)$-dimensional unit
sphere, fix $o\in\mathbb{S}^{n+1}$, and denote by
$\mathbb{S}^{n+1}_+:=B_{\pi/2}(o)$ the open hemisphere centred at $o$. In
geodesic polar coordinates at $o$, the spherical metric is
\begin{equation}\label{eq:metric}
\bar g=d\rho^2+\sin^2\rho\,\sigma_{\mathbb{S}^n},
\end{equation}
where $\rho$ is the geodesic distance from $o$ and $\sigma_{\mathbb{S}^n}$ is
the standard metric on $\mathbb{S}^n$. Let $X_0\colon M\to\mathbb{S}^{n+1}_+$
be a smooth, closed, strictly convex hypersurface enclosing $o$, and let
$f\in C^\infty((0,\pi/2))\cap C^0([0,\pi/2))$ satisfy $f(0)=0$ and $f(r)>0$
for $r>0$. We consider the non-homogeneous curvature flow
\begin{equation}\label{eq:unnormalized-flow}
\begin{cases}
\dfrac{\partial X}{\partial t}=-f(\rho)\sigma_k^\alpha\,\nu,\\[2mm]
X(\cdot,0)=X_0,
\end{cases}
\end{equation}
where $\nu$ is the outward unit normal and $\sigma_k$ is the $k$-th elementary
symmetric function of the principal curvatures.

To the best of our knowledge, contracting flows of the form \eqref{eq:unnormalized-flow} in
the sphere have not previously been studied; the results below appear to be
new even for the model profile $f(\rho)=\rho^\beta$. They apply to every
$\alpha>0$, without an axial-symmetry or curvature-pinching assumption. Two
features of the spherical setting are essential. First, the sphere admits no
ambient homothety, so the normalization is performed on the radial function in
geodesic normal coordinates at the contraction point, rather than on the
embedding. Second, the positive ambient curvature enters the curvature
estimate with the favourable sign: the background-curvature term that is
unfavourable in hyperbolic space here supplies the coercive quadratic term
that preserves strict convexity. We take $r^\beta$ as the model profile near
the origin and treat the supercritical and critical orders separately.

The equation \eqref{eq:unnormalized-flow} is parabolic on the G\ringaccent{a}rding cone
$\Gamma_k^+=\{\kappa\in\mathbb{R}^n:\sigma_j(\kappa)>0,\ j=1,\dots,k\}$ \cite{Gar59}.
In this paper we work on the strictly convex cone $\Gamma_n^+$ and use the
inverse Weingarten map to obtain a quantitative lower bound for all normalized
principal curvatures.

\subsection{The profile conditions}
Throughout the paper $f(0)=0$, $f(r)>0$ for $r>0$, and $\beta\ge 1+k\alpha$.
In the supercritical case $\beta>1+k\alpha$, we write $f(r)=r^\beta+g(r)$ and
impose the finite-order flatness condition \eqref{eq:C0-supercritical-flatness-spherical}. In the critical
case $\beta=1+k\alpha$, the remainder $g$ is required to decay strictly faster
than the model term, as in \eqref{eq:C0-critical-remainder-spherical}. In both cases the structural
condition
\[
\bigl(f^{\frac{1}{1+k\alpha}}\bigr)''\ge0
\]
is the radial convexity input in the $C^2$ estimate.

\subsection{Main results}

\begin{theorem}\label{thm:main}
Let $1\le k\le n$, $\alpha>0$. Let $M_0\subset\mathbb{S}^{n+1}_+$ be a smooth,
closed, strictly convex hypersurface enclosing $o$. Let
\begin{equation}\label{eq:C0-supercritical-decomposition-spherical}
\beta>1+k\alpha,\qquad m:=\lfloor\beta\rfloor,\qquad g(r):=f(r)-r^\beta.
\end{equation}
Assume that
\begin{equation}\label{eq:C0-supercritical-flatness-spherical}
g\in C^{m+1}([0,\pi/2)),\qquad g'(0)=\cdots=g^{(m)}(0)=0,
\end{equation}
and
\begin{equation}\label{eq:main-profile-root-convexity}
\bigl(f^{\frac{1}{1+k\alpha}}\bigr)''\ge0,\qquad r\in(0,\pi/2).
\end{equation}
Then the solution of \eqref{eq:unnormalized-flow} exists smoothly for all $t\ge0$, remains
strictly convex, and converges to $o$ as $t\to\infty$. Moreover, its
normalized radial function
\[
\widetilde\rho(\theta,\tau)=\lambda(t(\tau))\,\rho(\theta,t(\tau))
\]
converges to $1$ in $C^\infty(\mathbb{S}^n)$ as $\tau\to\infty$.
\end{theorem}

\begin{theorem}\label{thm:main-critical}
Let $1\le k\le n$, $\alpha>0$, and suppose $n\ge2$. Let
$M_0\subset\mathbb{S}^{n+1}_+$ be a smooth, closed, strictly convex
hypersurface enclosing $o$. Set
\begin{equation}\label{eq:C0-critical-decomposition-spherical}
\beta=1+k\alpha,\qquad g(r):=f(r)-r^{1+k\alpha}.
\end{equation}
Assume that there exists $\delta>0$ such that
\begin{equation}\label{eq:C0-critical-remainder-spherical}
g(r)=O\bigl(r^{1+k\alpha+\delta}\bigr)\quad\text{as }r\downarrow0.
\end{equation}
Set
\begin{equation}\label{eq:C0-critical-regularity-order-spherical}
N:=\lceil\beta+\delta\rceil
\end{equation}
and assume that $g\in C^N([0,\pi/2))$. Assume moreover that
\begin{equation}\label{eq:critical-profile-root-convexity}
\bigl(f^{\frac{1}{1+k\alpha}}\bigr)''\ge0,\qquad r\in(0,\pi/2).
\end{equation}
Then the solution of \eqref{eq:unnormalized-flow}  exists smoothly for all $t\ge0$, remains
strictly convex, and converges to $o$ as $t\to\infty$. Moreover, there exists
a constant $R_\infty>0$ such that
\[
\widetilde\rho(\cdot,\tau)\longrightarrow R_\infty\quad\text{in }C^\infty(\mathbb{S}^n)
\]
as $\tau\to\infty$. In general, $R_\infty$ depends on the initial hypersurface
and need not equal $1$.
\end{theorem}

\begin{remark}[The Gauss curvature case]\label{rem:gauss-curvature-case}
When $k=n$, one has $\Gamma_n^+=\Gamma_n$ and $\sigma_n(h)=K$, the Gauss
curvature. Theorems \ref{thm:main} and \ref{thm:main-critical} therefore
give a complete convex theory for the radially non-homogeneous Gauss curvature
flow
\[
\partial_tX=-f(\rho)K^\alpha\nu,\qquad \alpha>0,
\]
and no additional star-shapedness assumption is required; the geometric
justification is given at the beginning of Subsection 2.2. The determinant
structure is self-dual: for every positive definite endomorphism $A$,
\[
\sigma_n(A^{-1})=\sigma_n(A)^{-1},
\]
so for $F=\sigma_n^\alpha$ the dual function $F_*(A):=F(A^{-1})^{-1}$ agrees
with $F(A)$. Although $F$ is concave only when $\alpha\le1/n$, the
inverse-Weingarten argument used here applies for every $\alpha>0$, since its
inverse-concavity input is applied only to $\sigma_n^{1/n}$. Finally, the
smooth convergence $\widetilde\rho(\cdot,\tau)\to R_\infty$ (with $R_\infty=1$ in
the supercritical case) and the radial graph formula yield
\[
\hat\kappa^j_i\longrightarrow R_\infty^{-1}\delta^j_i,\qquad
\sigma_n(\hat h)\longrightarrow R_\infty^{-n},\qquad
\frac{\hat\kappa_{\max}}{\hat\kappa_{\min}}\longrightarrow1,
\]
so the normalized hypersurfaces become round also at the level of their
principal curvatures.
\end{remark}

\subsection{Outline of the proof}
The normalization is performed at the level of the radial function in normal
coordinates at $o$. After recording the radial graph geometry and the
normalized evolution equations (Section 2), we prove uniform $C^0$ and $C^1$
estimates and a positive lower bound for the normalized speed
(Sections 3--5). The central step is the $C^2$ estimate: we apply the maximum
principle, in the barrier sense, to the largest eigenvalue of the inverse
normalized Weingarten map. Inverse concavity controls the gradient terms,
while the positive ambient curvature supplies the favourable quadratic term
that preserves strict convexity (Lemma \ref{lem:convexity-preservation-spherical}). The speed upper bound and the
complete curvature bound then follow (Lemma \ref{lem:normalized-C2-estimate-spherical}). Finally, a
maximum-principle argument gives exponential decay of the normalized gradient
(Section 6), and Evans--Krylov theory \cite{Krylov87}, Schauder estimates, and
interpolation yield smooth exponential convergence to a geodesic sphere
centred at $o$ (Section 7).

\subsection*{Notation}

We write $\Phi=f(\rho)\sigma_k^\alpha(h)$ and
$F=\sigma_k^\alpha$, the object of action depends on the context ($h$ or $\hat{h}$). The normalized radial function, Weingarten map,
speed, and support function are denoted by $\widetilde\rho$, $\hat h$,
$\widetilde\Phi$, and $\widetilde u$, respectively.  Constants $c,C$
may change from line to line and depend only on the fixed data of the
flow and the initial hypersurface.

\section{Preliminaries}
\subsection{Radial graphs in
\texorpdfstring{$\mathbb{S}^{n+1}_+$}{a hemisphere}}
We first record some basic geometric quantities of a radial graph in
$\mathbb S^{n+1}_+$; see, e.g., \cite{Ger06} for the standard
computations in space forms. Let
$X(x)=(\rho(x),x)$, $x\in\mathbb S^n$, where
$\rho:\mathbb S^n\to(0,\pi/2)$ is smooth. All derivatives and
contractions below are taken with respect to $\sigma_{\mathbb S^n}$.

The induced metric and its inverse are
\begin{align}
  g_{ij}
  &=\rho_i\rho_j+\sin^2\rho\,\sigma_{ij},
    \label{eq:induced-metric}\\
  g^{ij}
  &=\frac{1}{\sin^2\rho}
    \left(
      \sigma^{ij}
      -\frac{\rho^i\rho^j}{\sin^2\rho\,v^2}
     \right).
     \label{eq:inverse-induced-metric}
\end{align}
The outward unit normal is
\begin{equation}\label{eq:radial-unit-normal}
  \nu=\frac{1}{v}
  \left(
    \partial_\rho-\frac{\rho^i}{\sin^2\rho}\partial_i
  \right),
\end{equation}
where
\begin{equation}\label{eq:radial-graph-factor}
  v=\left(1+\frac{|\nabla\rho|^2}{\sin^2\rho}\right)^{1/2}.
\end{equation}
With the convention $h_{ij}=-\langle\vec h_{ij},\nu\rangle$, the second
fundamental form is
\begin{equation}\label{eq:second-fundamental-form-rho}
  h_{ij}=\frac{1}{v}
  \left(
    -\rho_{ij}
    +\sin\rho\cos\rho\,\sigma_{ij}
    +2\cot\rho\,\rho_i\rho_j
  \right).
\end{equation}
The Weingarten map $h_i{}^j=g^{jk}h_{ik}$ is
\begin{equation}\label{eq:weingarten-map-rho}
  h_i{}^j
  =\frac{1}{v\sin^2\rho}
  \left(
    -\rho_i{}^j
    +\sin\rho\cos\rho\,\delta_i{}^j
    +\frac{\rho^j\rho^k\rho_{ki}}
           {\sin^2\rho\,v^2}
    +\frac{\cot\rho}{v^2}\rho_i\rho^j
  \right).
\end{equation}
The spherical support function is
\begin{equation}\label{eq:radial-support-function}
  u=\langle\sin\rho\,\partial_\rho,\nu\rangle
   =\frac{\sin\rho}{v}.
\end{equation}
\subsection{The normalized flow}\label{subsec:normalized-flow}
We first explain why no separate star-shapedness hypothesis is needed.
A smooth, closed, strictly convex hypersurface in
$\mathbb S^{n+1}_+$ enclosing $o$ is automatically a radial graph over
$\mathbb S^n$ centred at $o$.  Indeed, for $n\geq2$ the
do~Carmo--Warner theorem~\cite{dCW} shows that such a hypersurface is
embedded and bounds a strictly convex body
$\Omega\Subset\mathbb S^{n+1}_+$.  For $n=1$ the same conclusion is
elementary once ``enclosing $o$'' is understood to include embeddedness,
namely, $M_0$ is the boundary of a domain in $\mathbb S^2_+$ containing
$o$.  If a geodesic ray from $o$ met $\partial\Omega$ in more than one
point, or tangentially, the totally geodesic supporting hyperplane at
the last intersection point would separate $o$ from $\Omega$, contrary
to $o\in\operatorname{int}\Omega$.  Thus $M_t$ is star-shaped with
respect to $o$ as long as it remains strictly convex and encloses $o$.
Hence, after a time-dependent tangential
reparametrization, $M_t$ can be written as
\begin{equation}\label{eq:time-dependent-radial-graph}
  X(\theta,t)=\exp_o\bigl(\rho(\theta,t)\theta\bigr),
  \qquad \theta\in\mathbb S^n.
\end{equation}
For this parametrization,
\begin{equation}\label{eq:radial-velocity-normal-component}
  \left\langle\partial_tX,\nu\right\rangle
  =\rho_t\left\langle\partial_\rho,\nu\right\rangle
  =\frac{\rho_t}{v}.
\end{equation}
The tangential reparametrization does not change the normal velocity.
Comparing \eqref{eq:radial-velocity-normal-component} with
\eqref{eq:unnormalized-flow}, the geometric flow is equivalent to the
scalar equation on $\mathbb S^n$
\begin{equation}\label{eq:radial-flow}
  \begin{cases}
    \displaystyle
    \partial_t\rho
    =-f(\rho)\sigma_k^\alpha\bigl(h[\rho]\bigr)v,
       &\text{on }\mathbb S^n\times[0,T),\\[1mm]
    \rho(\cdot,0)=\rho_0.
  \end{cases}
\end{equation}
Here $h[\rho]$ is the Weingarten map in
\eqref{eq:weingarten-map-rho}.  Equation \eqref{eq:radial-flow} is
parabolic on the admissible branch $h[\rho]\in\Gamma_k^+$.

Spherical space has no ambient homothety.  Consequently, the
normalization is defined by blowing up the radial function in normal
coordinates at $o$, rather than by multiplying the embedding $X$.
Define
\begin{equation}\label{eq:lambda-definition}
  \lambda(t)=
  \begin{cases}
    e^{\gamma_0 t},&\beta=1+k\alpha,\\[1mm]
    \bigl(1+(\beta-k\alpha-1)\gamma_0 t\bigr)^%
      {1/(\beta-k\alpha-1)},&\beta>1+k\alpha,
  \end{cases}
\end{equation}
where $\gamma_0 = \binom nk^\alpha$ so that
\begin{equation}\label{eq:lambda-ode}
  \lambda'=\gamma_0\lambda^{2+k\alpha-\beta}.
\end{equation}
Introduce the normalized time
\begin{equation}\label{eq:normalized-time}
  \tau=
  \begin{cases}
    t,&\beta=1+k\alpha,\\[1mm]
    \displaystyle
    \frac{\log\bigl(1+(\beta-k\alpha-1)\gamma_0 t\bigr)}
    {(\beta-k\alpha-1)\gamma_0},&\beta>1+k\alpha.
  \end{cases}
\end{equation}
Equivalently,
\begin{equation}\label{eq:time-and-lambda-identities}
  \frac{d\tau}{dt}=\lambda^{1+k\alpha-\beta},
  \qquad
  \lambda=e^{\gamma_0\tau}.
\end{equation}
The normalized radial function and the corresponding auxiliary radial graph are
\begin{equation}\label{eq:normalized-radius}
  \widetilde\rho(\theta,\tau)=\lambda(t(\tau))\rho(\theta,t(\tau)),
  \qquad
  \widetilde M_\tau
  =\bigl\{\exp_o(\widetilde\rho(\theta,\tau)\theta):
      \theta\in\mathbb S^n\bigr\}.
\end{equation}
In what follows, we use only the normalized radial function
$\widetilde\rho$ and do not consider the geometry of the auxiliary graph
$\widetilde M_\tau$, since the curvatures of $\widetilde M_\tau$ and
$M_{t(\tau)}$ are not related by scaling.  

Differentiating $\widetilde\rho=\lambda\rho$ with respect to $t$ and
using \eqref{eq:radial-flow}, we obtain
\begin{equation}\label{eq:t-derivative-before-curvature-scaling}
  \begin{aligned}
  \partial_t\widetilde\rho
  &=(\partial_t\lambda)\rho+\lambda\partial_t\rho\\
  &=\gamma_0\lambda^{2+k\alpha-\beta}\rho
    -\lambda f
      \left(\frac{\widetilde\rho}{\lambda}\right)
      \sigma_k^\alpha
      \left(h\left[\frac{\widetilde\rho}{\lambda}\right]\right)v\\
  &=\gamma_0\lambda^{1+k\alpha-\beta}\widetilde\rho
    -\lambda f
      \left(\frac{\widetilde\rho}{\lambda}\right)
       \sigma_k^\alpha
       \left(h\left[\frac{\widetilde\rho}{\lambda}\right]\right)
       \left(
         1+\frac{|\nabla\widetilde\rho|^2}
           {\lambda^2\sin^2(\widetilde\rho/\lambda)}
      \right)^{1/2}.
  \end{aligned}
\end{equation}
Combining \eqref{eq:radial-flow},
\eqref{eq:time-and-lambda-identities}, and \eqref{eq:t-derivative-before-curvature-scaling},
we obtain
\begin{equation}\label{eq:normalized-radial-flow-exact}
  \begin{aligned}
  \partial_\tau\widetilde\rho
  &=\frac{dt}{d\tau}\partial_t\widetilde\rho
   =\lambda^{\beta-1-k\alpha}\partial_t\widetilde\rho\\
  &=\gamma_0\widetilde\rho
   -\lambda^{\beta-k\alpha} f
       \left(\frac{\widetilde\rho}{\lambda}\right)
     \sigma_k^\alpha
       \left(h\left[\frac{\widetilde\rho}{\lambda}\right]\right)
     \left(
       1+\frac{|\nabla\widetilde\rho|^2}
         {\lambda^2\sin^2(\widetilde\rho/\lambda)}
    \right)^{1/2}.
  \end{aligned}
\end{equation}

Define
\begin{equation}\label{eq:normalized-weingarten-map-definition}
  \hat h_i{}^j:=\lambda^{-1}h_i{}^j.
\end{equation}
Thus $\hat h_i{}^j$ is the rescaled Weingarten map of
$M_{t(\tau)}$; it is not the Weingarten map of the auxiliary graph
$\widetilde M_\tau$.  By the homogeneity of $\sigma_k$,
\[
  \sigma_k^\alpha
  \left(h\left[\frac{\widetilde\rho}{\lambda}\right]\right)
  =\lambda^{k\alpha}\sigma_k^\alpha(\hat h).
\]
Consequently, the normalized radial function satisfies
\begin{equation}\label{eq:normalized-radial-flow-hat-h}
  \begin{aligned}
  \partial_\tau\widetilde\rho
  ={}&\gamma_0\widetilde\rho
   -\lambda^\beta f
      \left(\frac{\widetilde\rho}{\lambda}\right)
     \sigma_k^\alpha(\hat h)
     \left(
       1+\frac{|\nabla\widetilde\rho|^2}
         {\lambda^2\sin^2(\widetilde\rho/\lambda)}
     \right)^{1/2}.
  \end{aligned}
\end{equation}
\subsection{The evolution equations}
In this subsection, we record the basic evolution equations for the
normalized Weingarten map, the normalized speed, and the inverse
normalized Weingarten map. Since the auxiliary graph $\widetilde M_\tau$ is not a
geometric rescaling of the original hypersurface, all geometric
quantities below are computed on the unnormalized flow $M_{t(\tau)}$.

\begin{lemma}[Evolution of the radial distance]
\label{lem:unnormalized-radial-distance-evolution-spherical}
Let $\bar\rho$ be the ambient radial distance from the origin and set
\[
  \rho(p,t):=\bar\rho(X(p,t)).
\]
Along the unnormalized flow~\eqref{eq:unnormalized-flow}, the radial
distance satisfies
\begin{equation}\label{eq:unnormalized-radial-distance-evolution-spherical}
  \partial_t\rho=-\frac{\Phi}{v}.
\end{equation}
\end{lemma}

\begin{proof}
Denote by $\bar\nabla$ the Levi--Civita connection of
$\mathbb S^{n+1}$. Then
$\bar\nabla\bar\rho=\partial_\rho$.

Since $\langle\partial_\rho,\nu\rangle=v^{-1}$, it follows
\[
  \partial_t\rho
  =\left\langle
      \overline\nabla\overline\rho,\partial_tX
    \right\rangle
  =-\Phi\langle\partial_\rho,\nu\rangle
  =-\frac{\Phi}{v}.
\]
\end{proof}

Define the normalized metric and speed by
\begin{equation}\label{eq:preliminary-normalized-metric-and-speed-spherical}
  \widetilde g:=\lambda^2g,
  \qquad
  \widetilde\Phi:=\lambda^{\beta-k\alpha}\Phi.
\end{equation}
The Levi--Civita connections of $g$ and $\widetilde g$ agree at each
fixed time because $\lambda$ is spatially constant. We write this connection as $\widetilde\nabla$.

\begin{lemma}[Evolution of the normalized Weingarten map]
\label{lem:basic-normalized-weingarten-evolution-spherical}
The normalized Weingarten map
$\hat h_i{}^j=\lambda^{-1}h_i{}^j$ satisfies
\begin{equation}\label{eq:basic-normalized-weingarten-evolution-spherical}
  \partial_\tau\hat h_i{}^j
  =\widetilde\nabla^j\widetilde\nabla_i\widetilde\Phi
   +\widetilde\Phi\hat h_i{}^l\hat h_l{}^j
   +\lambda^{-2}\widetilde\Phi\delta_i{}^j
   -\gamma_0\hat h_i{}^j.
\end{equation}
\end{lemma}

\begin{proof}
Along the unnormalized flow $\partial_tX=-\Phi\nu$ in
$\mathbb S^{n+1}$,
\begin{equation}\label{eq:preliminary-unnormalized-h-evolution-spherical}
  \partial_t h_i{}^j
  =\nabla^j\nabla_i\Phi
   +\Phi h_i{}^l h_l{}^j
   +\Phi\delta_i{}^j.
\end{equation}
Consequently,
\[
  \partial_t\hat h_i{}^j
  =-\frac{\lambda'}{\lambda}\hat h_i{}^j
   +\lambda^{-1}\nabla^j\nabla_i\Phi
   +\lambda^{-1}\Phi h_i{}^l h_l{}^j
   +\lambda^{-1}\Phi\delta_i{}^j.
\]
Using
\[
  h_i{}^j=\lambda\hat h_i{}^j,
  \qquad
  \Phi=\lambda^{k\alpha-\beta}\widetilde\Phi,
  \qquad
  \nabla^j\nabla_i\Phi
  =\lambda^{2+k\alpha-\beta}
   \widetilde\nabla^j\widetilde\nabla_i\widetilde\Phi,
\]
together with
\[
  \frac{\lambda'}{\lambda}
  =\gamma_0\lambda^{1+k\alpha-\beta},
  \qquad
  \frac{dt}{d\tau}=\lambda^{\beta-1-k\alpha},
\]
proves the assertion.
\end{proof}

For the speed evolution, put
\begin{equation}\label{eq:preliminary-speed-notation-spherical}
  F:=\sigma_k^\alpha(\hat h),
  \qquad
  \widetilde\Phi=\lambda^\beta f(\rho)F,
\end{equation}
and denote
\begin{equation}\label{eq:derivatives-of-sigma-k}
  \dot\sigma_k^{pq}
  :=\widetilde g^{pl}
    \frac{\partial\sigma_k}{\partial A_q{}^l}(\hat{h}),\quad
  \ddot\sigma_k^{pq,rs}
  :=\widetilde g^{pl}\widetilde g^{rm}
    \frac{\partial^2\sigma_k}
    {\partial A_q{}^l\partial A_s{}^m}(\hat{h}).
\end{equation}
Accordingly,
\[
  \dot F^{pq}
  =\alpha\sigma_k^{\alpha-1}\dot\sigma_k^{pq}.
\]
Define
\begin{equation}\label{eq:preliminary-linearized-operator-spherical}
  \mathcal L
  :=\partial_\tau
   -\lambda^\beta f(\rho)\dot F^{ij}
    \widetilde\nabla_i\widetilde\nabla_j.
\end{equation}

\begin{lemma}[Evolution of the normalized speed]
\label{lem:basic-normalized-speed-evolution-spherical}
The normalized speed satisfies
\begin{equation}\label{eq:basic-normalized-speed-evolution-spherical}
  \begin{aligned}
  \mathcal L\widetilde\Phi
  ={}&(\beta-k\alpha)\gamma_0\widetilde\Phi
     -\frac{f'(\rho)}{\lambda f(\rho)v}\widetilde\Phi^2+\lambda^\beta f(\rho)\widetilde\Phi\dot F^{ij}
      (\hat h^2)_{ij}
     +\lambda^{\beta-2}f(\rho)\widetilde\Phi
      \dot F^{ij}\widetilde g_{ij}.
  \end{aligned}
\end{equation}
\end{lemma}

\begin{proof}
We first calculate the evolution of $\Phi$ along~\eqref{eq:unnormalized-flow}.

By~\eqref{eq:unnormalized-radial-distance-evolution-spherical} and
\eqref{eq:preliminary-unnormalized-h-evolution-spherical},
\begin{equation}\label{eq:unnormalized-speed-evolution-spherical}
  \begin{aligned}
  \partial_t\Phi
  & =f'(\rho)\rho_t\sigma_k^\alpha(h)
     +f(\rho)
      \frac{\partial\sigma_k^\alpha}{\partial A_i{}^j}(h)
      \partial_t h_i{}^j\\
  &=-\frac{f'(\rho)}{f(\rho)v}\Phi^2
     +f(\rho)
      \frac{\partial\sigma_k^\alpha}{\partial A_i{}^j}(h)
      \left(
        \nabla^j\nabla_i\Phi
        +\Phi h_i{}^l h_l{}^j
        +\Phi\delta_i{}^j
      \right),
  \end{aligned}
\end{equation}
Recall $\widetilde\Phi=\lambda^{\beta-k\alpha}\Phi$.
Using \eqref{eq:time-and-lambda-identities}, we therefore obtain
\begin{equation}\label{eq:normalized-speed-chain-rule-spherical}
  \partial_\tau\widetilde\Phi
  =(\beta-k\alpha)\gamma_0\widetilde\Phi
   +\lambda^{2\beta-1-2k\alpha}\partial_t\Phi.
\end{equation}

Since
\[
  h_i{}^j=\lambda\hat h_i{}^j,
  \qquad
  \Phi=\lambda^{k\alpha-\beta}\widetilde\Phi,
  \qquad
  \nabla^j\nabla_i\Phi
  =\lambda^{2+k\alpha-\beta}
    \widetilde\nabla^j\widetilde\nabla_i\widetilde\Phi,
\]
and
\[
  \frac{\partial\sigma_k^\alpha}{\partial A_i{}^j}(h)
  =\lambda^{k\alpha-1}
   \frac{\partial\sigma_k^\alpha}{\partial A_i{}^j}(\hat h),
\]
equation \eqref{eq:unnormalized-speed-evolution-spherical} becomes
\begin{equation}\label{eq:scaled-unnormalized-speed-evolution-spherical}
  \begin{aligned}
  \lambda^{2\beta-1-2k\alpha}\partial_t\Phi
  ={}&\lambda^\beta f(\rho)\dot F^{ij}
      \widetilde\nabla_i\widetilde\nabla_j\widetilde\Phi
     -\frac{f'(\rho)}{\lambda f(\rho)v}\widetilde\Phi^2\\
   &+\lambda^\beta f(\rho)\widetilde\Phi\dot F^{ij}
      (\hat h^2)_{ij}
     +\lambda^{\beta-2}f(\rho)\widetilde\Phi
      \dot F^{ij}\widetilde g_{ij}.
  \end{aligned}
\end{equation}
Substituting
\eqref{eq:scaled-unnormalized-speed-evolution-spherical} into
\eqref{eq:normalized-speed-chain-rule-spherical} proves
\eqref{eq:basic-normalized-speed-evolution-spherical}.
\end{proof}

\begin{lemma}[Evolution of the normalized support function]
\label{lem:basic-normalized-support-evolution-spherical}
For
\begin{equation}\label{eq:preliminary-normalized-support-spherical}
  \widetilde u:=\lambda u
  =\frac{\lambda\sin\rho}{v},
\end{equation}
we have
\begin{equation}\label{eq:basic-normalized-support-evolution-spherical}
  \begin{aligned}
  \mathcal L\widetilde u
  ={}&\gamma_0\widetilde u
   -(1+k\alpha)\cos\rho\,\widetilde\Phi
   +\lambda^\beta f(\rho)\widetilde u\dot F^{ij}
      (\hat h^2)_{ij}\\
  &+\lambda^\beta f'(\rho)F\sin\rho\,(1-v^{-2}).
  \end{aligned}
\end{equation}
\end{lemma}

\begin{proof}
Recall
\[
  u=\langle V,\nu\rangle,
  \qquad
  V=\sin\rho\,\partial_\rho,
  \qquad
  \overline\nabla V=\cos\rho\,\mathrm{Id}.
\]
Define
\begin{equation}\label{eq:tangential-components-of-V-spherical}
  V_p:=\overline g(V,X_p),
  \qquad
  V^p:=g^{pq}V_q.
\end{equation}
Therefore,
\begin{equation}\label{eq:first-derivative-support-spherical}
  \begin{aligned}
  \nabla_i u
  &=\left\langle\overline\nabla_{X_i}V,\nu\right\rangle
    +\left\langle V,\overline\nabla_{X_i}\nu\right\rangle\\
  &=h_i{}^pV_p,
  \end{aligned}
\end{equation}
and
\[
  \nabla_j\nabla_i u
  =V^p\nabla_p h_{ij}
   +\cos\rho\,h_{ij}
   -u(h^2)_{ij}.
\]
We next calculate the time derivative. Along the unnormalized flow~\eqref{eq:unnormalized-flow}, the unit normal satisfies
$\partial_t\nu=\nabla\Phi$. Therefore,
\begin{equation}\label{eq:unnormalized-support-time-derivative-spherical}
  \begin{aligned}
  \partial_tu
  &=\left\langle
      \overline\nabla_{\partial_tX}V,\nu
    \right\rangle
    +\left\langle V,\partial_t\nu\right\rangle\\
  &=-\Phi\cos\rho+V^p\nabla_p\Phi.
  \end{aligned}
\end{equation}
By
$$\nabla_p\Phi = f'(\rho)\n_p\rho \sigma_k^\alpha(h) + \alpha f(\rho)\sigma_k^{\alpha-1}\frac{\p\sigma_k}{\p A_i{ }^j}(h)\n_ph_i{ }^j,$$
we have
\begin{equation}\label{eq:unnormalized-support-evolution-expanded-spherical}
  \begin{aligned}
  \partial_tu
  ={}&\alpha f(\rho)\sigma_k^{\alpha-1}
      \frac{\partial\sigma_k}{\partial A_i{}^l}(h)
      g^{lj}\nabla_j\nabla_i u\\
   &-\alpha f(\rho)\sigma_k^{\alpha-1}
      \frac{\partial\sigma_k}{\partial A_i{}^l}(h)
      g^{lj}\cos\rho\,h_{ij}\\
   &+\alpha f(\rho)\sigma_k^{\alpha-1}
      \frac{\partial\sigma_k}{\partial A_i{}^l}(h)
      g^{lj}u(h^2)_{ij}
      -\Phi\cos\rho\\
   &+f'(\rho)\sigma_k^\alpha(h)
  \sin\rho\,(1-v^{-2})
  \end{aligned}
\end{equation}
where we used $V^p\nabla_p\rho=\sin\rho\,(1-v^{-2})$.

Since
\begin{equation*}\label{eq:support-normalization-relations-spherical}
  \begin{gathered}
  \widetilde\nabla=\nabla,
  \quad
  \widetilde u=\lambda u,
  \quad
  \hat h_{ij}=\lambda h_{ij},
  \quad
  (\hat h^2)_{ij}=(h^2)_{ij},\\
  \Phi=\lambda^{k\alpha-\beta}\widetilde\Phi,\quad
  \alpha\sigma_k^{\alpha-1}(h)
  \frac{\partial\sigma_k}{\partial A_i{}^l}(h)g^{lj}
  =\lambda^{1+k\alpha}\dot F^{ij}.
  \end{gathered}
\end{equation*}
It follows from
\eqref{eq:unnormalized-support-evolution-expanded-spherical} that
\begin{equation}\label{eq:normalized-support-time-derivative-spherical}
  \begin{aligned}
  \partial_\tau\widetilde u
  &=\lambda_\tau u
    +\lambda\frac{dt}{d\tau}\partial_tu\\
  ={}&\gamma_0\widetilde u
    +\lambda^\beta f(\rho)\dot F^{ij}
      \widetilde\nabla_j\widetilde\nabla_i\widetilde u
    -\lambda^\beta f(\rho)\cos\rho\,
      \dot F^{ij}\hat h_{ij}\\
   &+\lambda^\beta f(\rho)\widetilde u\dot F^{ij}
      (\hat h^2)_{ij}
    -\widetilde\Phi\cos\rho
    +\lambda^\beta f'(\rho)F\sin\rho\,(1-v^{-2})\\
  ={}&\lambda^\beta f(\rho)\dot F^{ij}
      \widetilde\nabla_i\widetilde\nabla_j\widetilde u
    +\gamma_0\widetilde u
    -(1+k\alpha)\cos\rho\,\widetilde\Phi\\
   &+\lambda^\beta f(\rho)\widetilde u\dot F^{ij}
      (\hat h^2)_{ij}
    +\lambda^\beta f'(\rho)F\sin\rho\,(1-v^{-2}).
  \end{aligned}
\end{equation}
Moving the Hessian term to the left-hand side proves
\eqref{eq:basic-normalized-support-evolution-spherical}.
\end{proof}
Applying Simons' identity to
commute the second covariant derivatives of the second fundamental
form, we obtain refined evolution equations for the normalized
Weingarten map and its inverse.
\begin{lemma}
\begin{equation}\label{eq:full-normalized-weingarten-evolution}
  \begin{aligned}
  \mathcal L\hat h_i{}^j
  ={}&
  \lambda^\beta f(\rho)
  \alpha\sigma_k^{\alpha-1}\ddot\sigma_k^{pq,rs}
  \widetilde\nabla^j\hat h_{pq}
  \widetilde\nabla_i\hat h_{rs}
  \\
  &+
  \lambda^\beta f(\rho)
  \alpha(\alpha-1)\sigma_k^{\alpha-2}
  \dot\sigma_k^{pq}\dot\sigma_k^{rs}
  \widetilde\nabla^j\hat h_{pq}
  \widetilde\nabla_i\hat h_{rs}
  \\
  &+
  \lambda^\beta f'(\rho)
  \alpha\sigma_k^{\alpha-1}\dot\sigma_k^{pq}
  \Bigl(
    \widetilde\nabla^j\rho\,
    \widetilde\nabla_i\hat h_{pq}
    +\widetilde\nabla_i\rho\,
    \widetilde\nabla^j\hat h_{pq}
  \Bigr)
  \\
  &+
  \lambda^\beta f'(\rho)
  \sigma_k^\alpha
  \widetilde\nabla^j\widetilde\nabla_i\rho+
  \lambda^\beta f''(\rho)
  \sigma_k^\alpha
  \widetilde\nabla^j\rho\,
  \widetilde\nabla_i\rho
  \\
  &+
  \lambda^\beta f(\rho)
  \alpha\sigma_k^{\alpha-1}
  \left(
    \dot\sigma_k^{pq}(\hat h^2)_{pq}
    -\lambda^{-2}\dot\sigma_k^{pq}\widetilde g_{pq}
  \right)\hat h_i{}^j
  \\
  &+
  \lambda^\beta f(\rho)
  (1-k\alpha)\sigma_k^\alpha(\hat h^2)_i{}^j
  +
  \lambda^{\beta-2} f(\rho)
  (1+k\alpha)\sigma_k^\alpha\delta_i{}^j
  -\gamma_0\hat h_i{}^j.
  \end{aligned}
\end{equation}
\end{lemma}

\begin{proof}
Recall that
\[
  \widetilde\Phi
  =\lambda^\beta f(\rho)\sigma_k^\alpha(\hat h).
\]
Applying the product and chain rules twice, we obtain
\begin{equation}\label{eq:lem25-expanded-speed-hessian-spherical}
  \begin{aligned}
  \widetilde\nabla^j\widetilde\nabla_i\widetilde\Phi
  ={}&
  \lambda^\beta f(\rho)
  \alpha\sigma_k^{\alpha-1}\dot\sigma_k^{pq}
  \widetilde\nabla^j\widetilde\nabla_i\hat h_{pq}
  +
  \lambda^\beta f(\rho)
  \alpha\sigma_k^{\alpha-1}\ddot\sigma_k^{pq,rs}
  \widetilde\nabla^j\hat h_{pq}
  \widetilde\nabla_i\hat h_{rs}
  \\
  &+
  \lambda^\beta f(\rho)
  \alpha(\alpha-1)\sigma_k^{\alpha-2}
  \dot\sigma_k^{pq}\dot\sigma_k^{rs}
  \widetilde\nabla^j\hat h_{pq}
  \widetilde\nabla_i\hat h_{rs}
  \\
  &+
  \lambda^\beta f'(\rho)
  \alpha\sigma_k^{\alpha-1}\dot\sigma_k^{pq}
  \Bigl(
    \widetilde\nabla^j\rho\,
    \widetilde\nabla_i\hat h_{pq}
    +\widetilde\nabla_i\rho\,
    \widetilde\nabla^j\hat h_{pq}
  \Bigr)
  \\
  &+
  \lambda^\beta f'(\rho)
  \sigma_k^\alpha
  \widetilde\nabla^j\widetilde\nabla_i\rho
  +
  \lambda^\beta f''(\rho)
  \sigma_k^\alpha
  \widetilde\nabla^j\rho\,
  \widetilde\nabla_i\rho.
  \end{aligned}
\end{equation}

For a hypersurface in $\mathbb S^{n+1}$, the standard commutation
formula for the second covariant derivatives of its second fundamental
form is
\begin{equation}\label{eq:lem25-unscaled-commutation-formula-spherical}
  \begin{aligned}
  \nabla^j\nabla_i h_{pq}
  ={}&\nabla_p\nabla_q h_i{}^j
      +h_i{}^j(h^2)_{pq}
      -h_{pq}(h^2)_i{}^j\\
     &+h_q{}^j(h^2)_{pi}
      -h_{pi}(h^2)_q{}^j\\
     &+\delta_i{}^j h_{pq}
      -g_{pq}h_i{}^j
      +\delta_q{}^j h_{pi}
      -g_{pi}h_q{}^j.
  \end{aligned}
\end{equation}
Since $\widetilde g=\lambda^2g$ and
$\hat h_i{}^j=\lambda^{-1}h_i{}^j$, we have, as $(0,2)$-tensors,
$\hat h_{ij}=\lambda h_{ij}$. Therefore, rescaling
\eqref{eq:lem25-unscaled-commutation-formula-spherical} gives
\begin{equation}\label{eq:lem25-scaled-commutation-formula-spherical}
  \begin{aligned}
  \widetilde\nabla^j\widetilde\nabla_i\hat h_{pq}
  ={}&\widetilde\nabla_p\widetilde\nabla_q\hat h_i{}^j
      +\hat h_i{}^j(\hat h^2)_{pq}
      -\hat h_{pq}(\hat h^2)_i{}^j\\
     &+\hat h_q{}^j(\hat h^2)_{pi}
      -\hat h_{pi}(\hat h^2)_q{}^j\\
     &+\lambda^{-2}\Bigl(
        \delta_i{}^j\hat h_{pq}
        -\widetilde g_{pq}\hat h_i{}^j
        +\delta_q{}^j\hat h_{pi}
        -\widetilde g_{pi}\hat h_q{}^j
      \Bigr).
  \end{aligned}
\end{equation}
Euler's identity gives
\begin{equation}\label{eq:lem25-euler-identity-spherical}
  \dot\sigma_k^{pq}\hat h_{pq}=k\sigma_k(\hat{h}).
\end{equation}
Moreover, since $\dot\sigma_k$ commutes with $\hat h$, the mixed
terms in \eqref{eq:lem25-scaled-commutation-formula-spherical} satisfy
\[
  \dot\sigma_k^{pq}
  \bigl(
    \hat h_q{}^j(\hat h^2)_{pi}
    -\hat h_{pi}(\hat h^2)_q{}^j
  \bigr)=0
\]
and
\[
  \dot\sigma_k^{pq}
  \bigl(
    \delta_q{}^j\hat h_{pi}
    -\widetilde g_{pi}\hat h_q{}^j
  \bigr)=0.
\]

Substituting \eqref{eq:lem25-scaled-commutation-formula-spherical} into
the first term on the right-hand side of
\eqref{eq:lem25-expanded-speed-hessian-spherical}, and then using
\eqref{eq:basic-normalized-weingarten-evolution-spherical}, we obtain~\eqref{eq:full-normalized-weingarten-evolution}.
\end{proof}

\begin{lemma}[Evolution of the inverse Weingarten map]
\label{lem:inverse-weingarten-evolution-spherical}
Whenever $\hat h_i{}^j$ is positive definite, let
$\check h_i{}^j$ denote its inverse and set
$(\check h^2)_i{}^j:=\check h_i{}^p\check h_p{}^j$. Then
\begin{equation}\label{eq:full-inverse-weingarten-evolution-spherical}
  \begin{aligned}
  \mathcal L\check h_i{}^j
  ={}&-
  \lambda^\beta f(\rho)
  \alpha\sigma_k^{\alpha-1}\ddot\sigma_k^{pq,rs}
  \check h_i{}^a\check h_b{}^j
  \widetilde\nabla^b\hat h_{pq}
  \widetilde\nabla_a\hat h_{rs}
  \\
  &-
  \lambda^\beta f(\rho)
  \alpha(\alpha-1)\sigma_k^{\alpha-2}
  \dot\sigma_k^{pq}\dot\sigma_k^{rs}
  \check h_i{}^a\check h_b{}^j
  \widetilde\nabla^b\hat h_{pq}
  \widetilde\nabla_a\hat h_{rs}
  \\
  &-2\lambda^\beta f(\rho)
  \alpha\sigma_k^{\alpha-1}\dot\sigma_k^{rs}
  \check h_i{}^a
  (\widetilde\nabla_r\hat h_a{}^b)\check h_b{}^p
  (\widetilde\nabla_s\hat h_p{}^q)\check h_q{}^j
  \\
  &-
  \lambda^\beta f'(\rho)
  \alpha\sigma_k^{\alpha-1}\dot\sigma_k^{pq}
  \check h_i{}^a\check h_b{}^j
  \Bigl(
    \widetilde\nabla^b\rho\,
    \widetilde\nabla_a\hat h_{pq}
    +\widetilde\nabla_a\rho\,
    \widetilde\nabla^b\hat h_{pq}
  \Bigr)
  \\
  &-
  \lambda^\beta f'(\rho)\sigma_k^\alpha
  \check h_i{}^a\check h_b{}^j
  \widetilde\nabla^b\widetilde\nabla_a\rho
  -
  \lambda^\beta f''(\rho)\sigma_k^\alpha
  \check h_i{}^a\check h_b{}^j
  \widetilde\nabla^b\rho\,
  \widetilde\nabla_a\rho
  \\
  &-
  \lambda^\beta f(\rho)
  \alpha\sigma_k^{\alpha-1}
  \left(
    \dot\sigma_k^{pq}(\hat h^2)_{pq}
    -\lambda^{-2}\dot\sigma_k^{pq}\widetilde g_{pq}
  \right)\check h_i{}^j
  \\
  &+
  \lambda^\beta f(\rho)
  (k\alpha-1)\sigma_k^\alpha\delta_i{}^j
  -
  \lambda^{\beta-2}f(\rho)
  (1+k\alpha)\sigma_k^\alpha(\check h^2)_i{}^j
  +\gamma_0\check h_i{}^j.
  \end{aligned}
\end{equation}
\end{lemma}

\begin{proof}
The defining identity for the inverse Weingarten map is
\begin{equation}\label{eq:lem26-inverse-weingarten-identity-spherical}
  \check h_i{}^p\hat h_p{}^j=\delta_i{}^j.
\end{equation}
Differentiating \eqref{eq:lem26-inverse-weingarten-identity-spherical}
with respect to $\tau$ and in the spatial directions gives
\begin{equation}\label{eq:lem26-first-derivatives-inverse-weingarten-spherical}
  \begin{aligned}
  \partial_\tau\check h_i{}^j
  &=-\check h_i{}^p
    (\partial_\tau\hat h_p{}^q)\check h_q{}^j,\\
  \widetilde\nabla_r\check h_i{}^j
  &=-\check h_i{}^p
    (\widetilde\nabla_r\hat h_p{}^q)\check h_q{}^j.
  \end{aligned}
\end{equation}
Differentiating the second identity once more, we obtain
\begin{equation}\label{eq:lem26-second-derivative-inverse-weingarten-spherical}
  \begin{aligned}
  \widetilde\nabla_s\widetilde\nabla_r\check h_i{}^j
  ={}&-\check h_i{}^p
      (\widetilde\nabla_s\widetilde\nabla_r\hat h_p{}^q)
      \check h_q{}^j\\
     &+\check h_i{}^a
      (\widetilde\nabla_s\hat h_a{}^b)\check h_b{}^p
      (\widetilde\nabla_r\hat h_p{}^q)\check h_q{}^j\\
     &+\check h_i{}^p
      (\widetilde\nabla_r\hat h_p{}^q)\check h_q{}^a
      (\widetilde\nabla_s\hat h_a{}^b)\check h_b{}^j.
  \end{aligned}
\end{equation}
Since the coefficient
$\lambda^\beta f(\rho)\alpha\sigma_k^{\alpha-1}
\dot\sigma_k^{rs}$ is symmetric in $r,s$, contracting
\eqref{eq:lem26-second-derivative-inverse-weingarten-spherical} and
using \eqref{eq:lem26-first-derivatives-inverse-weingarten-spherical}
gives
\begin{equation}\label{eq:lem26-inverse-weingarten-operator-identity-spherical}
  \begin{aligned}
  \mathcal L\check h_i{}^j
  ={}&-\check h_i{}^p
      (\mathcal L\hat h_p{}^q)\check h_q{}^j\\
     &-2\lambda^\beta f(\rho)
       \alpha\sigma_k^{\alpha-1}\dot\sigma_k^{rs}
       \check h_i{}^a
       (\widetilde\nabla_r\hat h_a{}^b)\check h_b{}^p
       (\widetilde\nabla_s\hat h_p{}^q)\check h_q{}^j.
  \end{aligned}
\end{equation}
Finally, substitute \eqref{eq:full-normalized-weingarten-evolution}
into \eqref{eq:lem26-inverse-weingarten-operator-identity-spherical}
and use
\[
  \check h_i{}^a\hat h_a{}^b\check h_b{}^j
  =\check h_i{}^j,
  \qquad
  \check h_i{}^a(\hat h^2)_a{}^b\check h_b{}^j
  =\delta_i{}^j,
  \qquad
  \check h_i{}^a\delta_a{}^b\check h_b{}^j
  =(\check h^2)_i{}^j.
\]
Collecting the resulting terms proves
\eqref{eq:full-inverse-weingarten-evolution-spherical}.
\end{proof}
\section{The \texorpdfstring{$C^0$}{C0} estimate}
In this section, we derive the $C^0$ estimate for the normalized
flow~\eqref{eq:normalized-radial-flow-hat-h}. We first derive a useful growth estimate on $f$.
\begin{lemma}
\label{lem:model-profile-comparison-spherical}
Assume that either
\eqref{eq:C0-supercritical-decomposition-spherical} and
\eqref{eq:C0-supercritical-flatness-spherical} hold, or
\eqref{eq:C0-critical-decomposition-spherical} and
\eqref{eq:C0-critical-remainder-spherical} hold.  Then, for every
$R_0\in(0,\pi/2)$, there exist constants $0<c\leq C<\infty$,
depending only on $f$, such that
\begin{equation}\label{eq:model-profile-two-sided-spherical}
  cr^\beta\leq f(r)\leq Cr^\beta,
  \qquad 0\leq r\leq R_0.
\end{equation}
\end{lemma}

\begin{proof}
Suppose first that
\eqref{eq:C0-supercritical-decomposition-spherical} and
\eqref{eq:C0-supercritical-flatness-spherical} hold.  Then
$m=\lfloor\beta\rfloor$ and $g(0)=0$.
Since $g\in C^{m+1}$ and $g^{(j)}(0)=0$ for $j=0,\ldots,m$,
Taylor's theorem gives
\[
  g(r)=o(r^{\beta}).
\]
If instead \eqref{eq:C0-critical-decomposition-spherical} and
\eqref{eq:C0-critical-remainder-spherical} hold, then
\eqref{eq:C0-critical-remainder-spherical} directly gives
\[
  \frac{g(r)}{r^\beta}=O(r^\delta)\longrightarrow0.
\]
Thus, in either case,
\begin{equation*}
  \lim_{r\downarrow0}\frac{f(r)}{r^\beta}=1.
\end{equation*}
Therefore, the ratio $\frac{f(r)}{r^\beta}$ is a positive continuous
function on $[0,R_0]$. Hence it attains a positive minimum and a
finite maximum on this interval.
\end{proof}

\begin{lemma}
\label{lem:critical-profile-spherical}
Assume that \eqref{eq:C0-critical-decomposition-spherical} and
\eqref{eq:C0-critical-remainder-spherical} hold.  For
$r\in(0,\pi/2)$, define
\begin{equation}\label{eq:critical-profile-factor-spherical}
  A(r):=\frac{f(r)}{r}\cot^{k\alpha}r
\end{equation}
and
\begin{equation}
  E(r) := A(r) - 1.
\end{equation}
Set $\delta_0:=\min\{\delta,2\}>0$. For every
$R_0\in(0,\pi/2)$, there exist constants
$0<a\leq b<\infty$ and $C>0$, depending only on $f$ and $R_0$, such
that
\begin{equation}\label{eq:critical-profile-error-spherical}
  a\leq A(r)\leq b,
  \qquad
  |E(r)|\leq Cr^{\delta_0},
  \qquad 0<r\leq R_0.
\end{equation}
Consequently, if $z:[0,T)\to(0,R_0]$ satisfies
\[
  0<z(t)\leq C_1e^{-c_1t},
  \qquad 0\leq t<T,
\]
for some $C_1,c_1>0$, then
\begin{equation}\label{eq:critical-profile-composition-integrable-spherical}
  \int_0^T |E(z(t))|\,dt
  \leq\frac{CC_1^{\delta_0}}{c_1\delta_0}.
\end{equation}
\end{lemma}

\begin{proof}
The bounds for $A$ follow directly from
Lemma~\ref{lem:model-profile-comparison-spherical} and the positive
lower and upper bounds for $r\cot r$ on $(0,R_0]$. By~\eqref{eq:C0-critical-decomposition-spherical}, we have
$$E(r) = r^{k\alpha}\cot^{k\alpha}r-1 + \frac{g(r)}{r}\cot^{k\alpha}r.$$
Hence,
\begin{equation*}
  \begin{split}
     \frac{|E(r)|}{r^{\delta_0}}\leq{} & \frac{|r^{k\alpha}\cot^{k\alpha}r-1|}{r^{\delta_0}} + \frac{|g(r)|}{r^{1+\delta_0}}\cot^{k\alpha}r \\
       \leq &  Cr^{2-\delta_0} + Cr^{k\alpha+\delta-\delta_0}\cot^{k\alpha}r
  \end{split}
\end{equation*}
where we used~\eqref{eq:C0-critical-remainder-spherical} and
$r^{k\alpha}\cot^{k\alpha}r=1+O(r^2)$. Therefore, by our choice of
$\delta_0$, the ratio $\frac{|E(r)|}{r^{\delta_0}}$ is bounded on the
interval $(0,R_0]$.

Finally,
\[
  |E(z(t))|
  \leq CC_1^{\delta_0}e^{-c_1\delta_0t},
\]
which proves \eqref{eq:critical-profile-composition-integrable-spherical}.
\end{proof}

We formulate the $C^0$ estimate as follows.

\begin{lemma}[$C^0$ estimate]
\label{lem:normalized-C0-estimate-spherical}
Under the assumptions of either Theorem~\ref{thm:main} or
Theorem~\ref{thm:main-critical}, there exist positive constants $c_0$ and
$C_0$, depending only on the initial radial bounds and $f$, such that
\begin{equation}\label{eq:normalized-C0-estimate-spherical}
  0<c_0\leq\widetilde\rho(\theta,\tau)\leq C_0.
\end{equation}
\end{lemma}

\begin{proof}
The scalar equation \eqref{eq:radial-flow} gives
$\partial_t\rho<0$ and hence $0<\rho(\theta,t)\leq R_0$ with $R_0$ the initial radial bound.

\begin{itemize}
  \item Suppose first that
  \eqref{eq:C0-supercritical-decomposition-spherical} and
  \eqref{eq:C0-supercritical-flatness-spherical} hold.  Define
  \begin{equation}\label{eq:normalized-radial-extrema-spherical}
    \widetilde\rho_-(\tau)
    :=\min_{\mathbb S^n}\widetilde\rho(\cdot,\tau),
    \qquad
    \widetilde\rho_+(\tau)
    :=\max_{\mathbb S^n}\widetilde\rho(\cdot,\tau).
  \end{equation}
  At a spatial maximum of $\widetilde\rho$, one has $v=1$ and
  $\nabla^2\widetilde\rho\leq0$.  Since $\lambda$ is spatially
  constant, \eqref{eq:weingarten-map-rho} and
  $\hat h_i{}^j=\lambda^{-1}h_i{}^j$ give
  \begin{equation}\label{eq:normalized-maximum-curvature-comparison-spherical}
    \hat h_i{}^j
    \geq\lambda^{-1}\cot
      \frac{\widetilde\rho_+}{\lambda}\,\delta_i{}^j.
  \end{equation}
  Similarly, at a spatial minimum,
  \begin{equation}\label{eq:normalized-minimum-curvature-comparison-spherical}
    \hat h_i{}^j
    \leq\lambda^{-1}\cot
      \frac{\widetilde\rho_-}{\lambda}\,\delta_i{}^j.
  \end{equation}
  Applying the maximum principle directly to
  \eqref{eq:normalized-radial-flow-hat-h} and setting
  $\mu:=\beta-k\alpha-1>0$, we obtain
  \begin{equation}\label{eq:normalized-extrema-inequalities-spherical}
    \begin{aligned}
    \frac{d}{d\tau}\widetilde\rho_+
    &\leq\gamma_0\widetilde\rho_+
     -\gamma_0\lambda^{1+\mu}
      f\left(\frac{\widetilde\rho_+}{\lambda}\right)
      \cot^{k\alpha}\left(\frac{\widetilde\rho_+}{\lambda}\right),\\
    \frac{d}{d\tau}\widetilde\rho_-
    &\geq\gamma_0\widetilde\rho_-
     -\gamma_0\lambda^{1+\mu}
      f\left(\frac{\widetilde\rho_-}{\lambda}\right)
      \cot^{k\alpha}\left(\frac{\widetilde\rho_-}{\lambda}\right).
    \end{aligned}
  \end{equation}
  These inequalities hold at every differentiability time of the
  extrema, equivalently in the barrier sense.  By
  \eqref{eq:model-profile-two-sided-spherical} and the
  bounds for $r\cot r$ on $(0,R_0]$, there exist constants
  $0<a_0<a_1$, depending only on $f$ and $R_0$, such that
\begin{equation}\label{eq:normalized-logistic-coefficient-spherical}
  a_0\leq
  \frac{f(r)}{r^\beta}(r\cot r)^{k\alpha}
  \leq a_1,
  \qquad 0<r\leq R_0.
\end{equation}
Consequently, \eqref{eq:normalized-extrema-inequalities-spherical} becomes
\begin{equation}\label{eq:normalized-logistic-inequalities-spherical}
  \begin{aligned}
  \frac{d}{d\tau}\widetilde\rho_+
  &\leq\gamma_0\widetilde\rho_+
       (1-a_0\widetilde\rho_+^\mu),\\
  \frac{d}{d\tau}\widetilde\rho_-
  &\geq\gamma_0\widetilde\rho_-
       (1-a_1\widetilde\rho_-^\mu).
  \end{aligned}
\end{equation}
Therefore the standard comparison principle gives
\begin{equation}\label{eq:normalized-C0-supercritical-spherical}
  \min\{\widetilde\rho_-(0),a_1^{-1/\mu}\}
  \leq\widetilde\rho(\theta,\tau)
  \leq\max\{\widetilde\rho_+(0),a_0^{-1/\mu}\}.
\end{equation}
  \item Suppose next that
  \eqref{eq:C0-critical-decomposition-spherical} and
  \eqref{eq:C0-critical-remainder-spherical} hold.  In this case
  $\mu=0$ and $\tau=t$.  Define the extrema of the unnormalized radial
  function by
  \[
    \rho_-(t):=\min_{\mathbb S^n}\rho(\cdot,t),
    \qquad
    \rho_+(t):=\max_{\mathbb S^n}\rho(\cdot,t).
  \]
  Then $\widetilde\rho_\pm=\lambda\rho_\pm$.  Let $A$ and $E$ be
  the functions defined in Lemma~\ref{lem:critical-profile-spherical}.
  
  At a spatial maximum of $\rho$, one has $\nabla\rho=0$ and
  $\nabla^2\rho\leq0$, and hence
  \[
    h_i{}^j[\rho]\geq\cot\rho_+\,\delta_i{}^j.
  \]
  Applying \eqref{eq:radial-flow} at every differentiability time of
  $\rho_+$ gives
  \begin{equation}\label{eq:critical-unnormalized-maximum-spherical}
    \frac{d}{dt}\rho_+
    \leq-\gamma_0 f(\rho_+)\cot^{k\alpha}\rho_+
    =-\gamma_0A(\rho_+)\rho_+.
  \end{equation}
  By Lemma~\ref{lem:critical-profile-spherical}, $A\geq a>0$ on
  $(0,R_0]$.  Thus the ODE comparison yields
  \begin{equation}\label{eq:critical-unnormalized-extrema-decay-spherical}
    0<\rho_+\leq\rho_+(0)e^{-a\gamma_0t}.
  \end{equation}
  The last assertion of Lemma~\ref{lem:critical-profile-spherical} now
  gives
  \begin{equation}\label{eq:critical-upper-error-integrable-spherical}
    \int_0^T |E(\rho_+(t))|\,dt\leq C,
  \end{equation}
  where $C$ is independent of $T$.
  Since $\lambda_t=\gamma_0\lambda$, equation
  \eqref{eq:critical-unnormalized-maximum-spherical} implies
  \begin{equation}\label{eq:critical-lambda-maximum-comparison-spherical}
    \frac{d}{dt}\log(\lambda\rho_+)
    \leq\gamma_0\bigl(1-A(\rho_+)\bigr)
    =-\gamma_0E(\rho_+)
    \leq\gamma_0|E(\rho_+)|.
  \end{equation}
  Integrating and using
  \eqref{eq:critical-upper-error-integrable-spherical}, we obtain
  \[
    \widetilde\rho_+(t)=\lambda(t)\rho_+(t)
    \leq\rho_+(0)
      \exp\left(
        \gamma_0\int_0^T|E(\rho_+(s))|\,ds
      \right)
    \leq C_0.
  \]

  We next estimate the minimum radius.  At a spatial minimum of
  $\rho$, one has $\nabla\rho=0$ and $\nabla^2\rho\geq0$, and hence
  \[
    h_i{}^j[\rho]\leq\cot\rho_-\,\delta_i{}^j.
  \]
  Therefore \eqref{eq:radial-flow} gives, at every differentiability
  time of $\rho_-$,
  \begin{equation}\label{eq:critical-unnormalized-minimum-spherical}
    \frac{d}{dt}\rho_-
    \geq-\gamma_0 f(\rho_-)\cot^{k\alpha}\rho_-
    =-\gamma_0A(\rho_-)\rho_-.
  \end{equation}
  Since $0<\rho_-\leq\rho_+$,
  \eqref{eq:critical-unnormalized-extrema-decay-spherical} and the last
  assertion of Lemma~\ref{lem:critical-profile-spherical} imply
  \begin{equation}\label{eq:critical-lower-error-integrable-spherical}
    \int_0^T |E(\rho_-(t))|\,dt\leq C,
  \end{equation}
  where $C$ is independent of $T$.
  Combining \eqref{eq:critical-unnormalized-minimum-spherical} with
  $\lambda_t=\gamma_0\lambda$, we obtain
  \begin{equation}\label{eq:critical-lambda-minimum-comparison-spherical}
    \frac{d}{dt}\log(\lambda\rho_-)
    \geq\gamma_0\bigl(1-A(\rho_-)\bigr)
    =-\gamma_0E(\rho_-)
    \geq-\gamma_0|E(\rho_-)|.
  \end{equation}
  Integrating and using
  \eqref{eq:critical-lower-error-integrable-spherical}, we obtain
  \[
    \widetilde\rho_-(t)=\lambda(t)\rho_-(t)
    \geq\rho_-(0)
      \exp\left(
        -\gamma_0\int_0^T|E(\rho_-(s))|\,ds
      \right)
    \geq c_0>0.
  \]
  Consequently,
  \[
    c_0\leq\widetilde\rho_-(t)
    \leq\widetilde\rho(\theta,t)
    \leq\widetilde\rho_+(t)\leq C_0,
  \]
  which completes the proof of
  \eqref{eq:normalized-C0-estimate-spherical} in the critical case.
\end{itemize}
\end{proof}

\begin{corollary}\label{cor:lambda-sin-rho-two-sided-spherical}
Under the assumptions of either Theorem~\ref{thm:main} or
Theorem~\ref{thm:main-critical},
there exist constants $c_1,C_1>0$ such that
\[
  c_1\leq\lambda\sin\rho\leq C_1.
\]
\end{corollary}

\begin{proof}
Since the hypersurface lies in the open hemisphere (i.e., $0<\rho<\pi/2$), it follows that
\[
  \frac{2}{\pi}\rho\leq\sin\rho\leq\rho.
\]
Multiplying by $\lambda$ and using
\eqref{eq:normalized-C0-estimate-spherical}, we obtain
\[
  \frac{2c_0}{\pi}
  \leq\lambda\sin\rho
  \leq C_0,
\]
which proves the assertion.
\end{proof}

\begin{remark}
For the $C^0$ estimate alone, the assumptions on $f$ in the
supercritical case can be weakened.  It is enough to assume
\[
  0<\liminf_{r\downarrow0}\frac{f(r)}{r^\beta}
  \leq
  \limsup_{r\downarrow0}\frac{f(r)}{r^\beta}
  <\infty.
\]
Indeed, since $f$ is positive and continuous away from the origin,
these asymptotic bounds give the two-sided estimate
\eqref{eq:model-profile-two-sided-spherical} on every fixed interval
$(0,R_0]$.

In the critical case $\beta=1+k\alpha$, the present argument requires
the stronger decomposition
\[
  f(r)=r^{1+k\alpha}+g(r),
  \qquad
  g(r)=O\bigl(r^{1+k\alpha+\delta}\bigr)
\]
for some $\delta>0$.  Thus the remainder $g$ must decay strictly faster
than the leading term $r^{1+k\alpha}$ as $r\downarrow0$.  This positive
power gain makes the error $E(\rho_\pm(t))$ integrable in time.  The
weaker condition $g=o(r^{1+k\alpha})$ alone does not, in general,
guarantee this integrability.
\end{remark}

\section{The \texorpdfstring{$C^1$}{C1} estimate}

The strict convexity assumption is imposed only on the initial
hypersurface. By continuity, there exists a maximal time interval
$[0,\tau_*)$ on which the evolving hypersurface remains strictly
convex. Recall the spherical support function $u=\sin\rho/v$ from
\eqref{eq:radial-support-function} and its normalized counterpart
$\widetilde u=\lambda u$ from
\eqref{eq:preliminary-normalized-support-spherical}.

\begin{lemma}[$C^1$ estimate on the maximal convexity interval]
\label{lem:normalized-C1-estimate-spherical}
Under the assumptions of either Theorem~\ref{thm:main} or
Theorem~\ref{thm:main-critical}, there exist constants $c,C>0$ depending only on the $C^0$ bounds of Lemma~\ref{lem:normalized-C0-estimate-spherical},
independent of $\tau_*$, such that on $\mathbb S^n\times [0,\tau_*)$,
\begin{equation}\label{eq:normalized-C1-estimate-spherical}
  v\leq C,
  \qquad
  \widetilde u\geq c,
  \qquad
  |\nabla\widetilde\rho|\leq C.
\end{equation}
\end{lemma}

\begin{proof}
Fix $\tau\in [0,\tau_*)$ and let $\theta_0$ be a minimum point of
$\widetilde u(\cdot,\tau)$.  Put
$V=\sin\rho\,\partial_\rho$.  Since
$\overline\nabla V=\cos\rho\,\mathrm{Id}$, differentiation of
$u=\langle V,\nu\rangle$ gives
\begin{equation}\label{eq:first-derivative-support-C1-spherical}
  \nabla_i u
  =h_i{}^j\langle V,X_j\rangle.
\end{equation}
Because $\lambda$ is spatially constant, at $\theta_0$ we have
\begin{equation}\label{eq:support-minimum-C1-spherical}
  0=\nabla_i\widetilde u
   =\lambda h_i{}^j\langle V,X_j\rangle.
\end{equation}
On $[0,\tau_*)$, the Weingarten map $h_i{}^j$ is positive definite and hence
invertible.  It follows that
\[
  \langle V,X_j\rangle=0
  \qquad\text{for every }j.
\]
For the radial graph,
\[
  \langle V,X_j\rangle=\sin\rho\,\rho_j.
\]
Since $0<\rho<\pi/2$, it follows that
$\rho_j(\theta_0,\tau)=0$ for every $j$. Consequently,
$v(\theta_0,\tau)=1$ and
\[
  \min_{\mathbb S^n}\widetilde u(\cdot,\tau)
  =\lambda\sin\rho(\theta_0,\tau).
\]
Corollary~\ref{cor:lambda-sin-rho-two-sided-spherical} now yields
\begin{equation}\label{eq:normalized-support-lower-bound-C1-spherical}
  \widetilde u\geq c
  \qquad\text{on }\mathbb S^n\times[0,\tau_*).
\end{equation}
The same corollary gives $\lambda\sin\rho\leq C$, and hence
\[
  v=\frac{\lambda\sin\rho}{\widetilde u}\leq C.
\]
Finally, the definition of $v$ implies
\[
  |\nabla\widetilde\rho|
  =\lambda|\nabla\rho|
  =\lambda\sin\rho\sqrt{v^2-1}\leq C.
\]
All constants depend only on the constants in the normalized $C^0$
estimate and are independent of $\tau_*$.
\end{proof}

\section{The \texorpdfstring{$C^2$}{C2} estimate}
\begin{lemma}\label{lem:profile-logarithmic-derivative-bound-spherical}
Assume that either
\eqref{eq:C0-supercritical-decomposition-spherical},
\eqref{eq:C0-supercritical-flatness-spherical}, and
\eqref{eq:main-profile-root-convexity} hold, or that
\eqref{eq:C0-critical-decomposition-spherical},
\eqref{eq:C0-critical-remainder-spherical}, and
\eqref{eq:critical-profile-root-convexity} hold.  Then, for every
$R_0\in(0,\pi/2)$, there exists a constant $C>0$, depending only on
$f$, such that
\begin{equation}\label{eq:profile-logarithmic-derivative-bound-spherical}
  0\leq\frac{r f'(r)}{f(r)}\leq C,
  \qquad 0<r\leq R_0.
\end{equation}
\end{lemma}

\begin{proof}
Set $q=f^{1/(1+k\alpha)}$. Since $q$ is
convex on $(0,\pi/2)$, $q(0)=0$ and $q(r)>0$ for $r>0$, it follows that 
\[
  q'(r)\geq\frac{q(r)-q(0)}{r}
  =\frac{q(r)}{r}>0,
\]
and hence $f'(r) > 0$ for $r > 0$.

Let $0<r\leq R_0/2$. Again by the convexity of $q$, we have
\[
  q'(r)\leq\frac{q(2r)-q(r)}{r}
  \leq\frac{q(2r)}{r}.
\]
Therefore,
\begin{equation}\label{eq:logarithmic-derivative-chord-bound-spherical}
  \frac{r f'(r)}{f(r)}
  =(1+k\alpha)\frac{r q'(r)}{q(r)}
  \leq(1+k\alpha)\frac{q(2r)}{q(r)}.
\end{equation}
By Lemma~\ref{lem:model-profile-comparison-spherical},
$cr^\beta\leq f(r)\leq Cr^\beta$ on $[0,R_0]$.  Hence
\[
  \frac{q(2r)}{q(r)}
  =\left(\frac{f(2r)}{f(r)}\right)^{\frac1{1+k\alpha}}
  \leq C.
\]
This proves the required upper bound for $0<r\leq R_0/2$.  On the
compact interval $[R_0/2,R_0]$, the smoothness and positivity of $f$
imply that $r f'(r)/f(r)$ is bounded.  This proves
\eqref{eq:profile-logarithmic-derivative-bound-spherical}.
\end{proof}

\begin{lemma}[Lower bound for the normalized speed]
\label{lem:normalized-speed-lower-bound-spherical}
On the maximal strictly convex interval $[0,\tau_*)$, there exists a
constant $c_{\widetilde\Phi}>0$, independent of $\tau_*$, such that
\begin{equation}\label{eq:normalized-speed-lower-bound-spherical}
  \widetilde\Phi\geq c_{\widetilde\Phi}.
\end{equation}
\end{lemma}

\begin{proof}
By Lemma~\ref{lem:model-profile-comparison-spherical}, the normalized
$C^0$ estimate, and $\rho=\widetilde\rho/\lambda$, we have
\[
  c\widetilde\rho^\beta
  \leq\lambda^\beta f(\rho)
  \leq C\widetilde\rho^\beta.
\]
Lemma~\ref{lem:profile-logarithmic-derivative-bound-spherical} also
gives
\[
  \frac{f'(\rho)}{\lambda f(\rho)}
  =\frac{\rho f'(\rho)}{f(\rho)}
    \frac{1}{\lambda\rho}
  =\frac{\rho f'(\rho)}{f(\rho)}
    \frac{1}{\widetilde\rho}.
\]
Since $0<c_0\leq\widetilde\rho\leq C_0$, it follows that
\begin{equation}\label{eq:speed-profile-coefficient-bounds}
  0<c\leq\lambda^\beta f(\rho)\leq C,
  \qquad
  0\leq\frac{f'(\rho)}{\lambda f(\rho)}\leq C.
\end{equation}

Let
\[
  \widetilde{\Phi}_{\min}(\tau):=\min_{\mathbb S^n}\widetilde\Phi(\cdot,\tau).
\]
At a spatial minimum of $\widetilde\Phi$, strict convexity and the
ellipticity of $F$ imply
\[
  \lambda^\beta f(\rho)\widetilde\Phi
    \dot F^{ij}(\hat h^2)_{ij}\geq0,
  \qquad
  \lambda^{\beta-2}f(\rho)\widetilde\Phi
    \dot F^{ij}\widetilde g_{ij}\geq0.
\]
Consequently, the speed evolution equation
\eqref{eq:basic-normalized-speed-evolution-spherical} and
\eqref{eq:speed-profile-coefficient-bounds} give
\begin{equation}\label{eq:minimum-speed-logistic-inequality-spherical}
  \frac{d}{d\tau}\widetilde{\Phi}_{\min}
  \geq(\beta-k\alpha)\gamma_0\widetilde{\Phi}_{\min}-C\widetilde{\Phi}_{\min}^2.
\end{equation}
Since $\beta-k\alpha\geq1$, comparison with the corresponding logistic
equation yields
\[
  \widetilde\Phi_{\min}(\tau)
  \geq
  \min\left\{
    \widetilde\Phi_{\min}(0),
    \frac{(\beta-k\alpha)\gamma_0}{C}
  \right\}>0.
\]
This proves \eqref{eq:normalized-speed-lower-bound-spherical}.
\end{proof}

\begin{lemma}\label{lem:profile-second-derivative-bound-spherical}
Assume that the conditions of either Theorem~\ref{thm:main} or
Theorem~\ref{thm:main-critical} hold.  Then, for every
$R_0\in(0,\pi/2)$, there exists a constant $C>0$, depending only on
$f$, such that
\begin{equation}\label{eq:profile-second-derivative-bound-spherical}
  |f''(r)|\leq Cr^{\beta-2},
  \qquad 0<r\leq R_0.
\end{equation}
\end{lemma}

\begin{proof}
Write $f(r)=r^\beta+g(r)$.  In the supercritical case, let
$m=\lfloor\beta\rfloor$.  The vanishing of
$g^{(j)}(0)$ for $j=0,\ldots,m$, together with
$g\in C^{m+1}$, allows us to apply Taylor's theorem to $g''$ and gives
\[
  g''(r)=O(r^{m-1})=o(r^{\beta-2}).
\]

In the critical case, the assumptions of
Theorem~\ref{thm:main-critical}, with
$N=\lceil\beta+\delta\rceil$, give
\[
  g^{(j)}(0)=0,
  \qquad j=0,\ldots,N-1.
\]
Taylor's theorem applied to $g''$ therefore yields
\[
  g''(r)=O(r^{N-2})=o(r^{\beta-2}),
\]
where we used $N>\beta$.

Thus, in both cases,
\[
  f''(r)=\beta(\beta-1)r^{\beta-2}+o(r^{\beta-2})
  \qquad\text{as }r\downarrow0.
\]
The estimate near the origin, together with the smoothness of $f$ on
every compact subinterval of $(0,\pi/2)$, proves
\eqref{eq:profile-second-derivative-bound-spherical}.
\end{proof}

Now we prove that the normalized Weingarten map has a positive lower
bound, which implies that convexity is preserved along the flow.  Let
$\hat\kappa_1\leq\cdots\leq\hat\kappa_n$ denote the eigenvalues of
$\hat h$.  By Lemmas~\ref{lem:model-profile-comparison-spherical},
\ref{lem:profile-logarithmic-derivative-bound-spherical}, and
\ref{lem:profile-second-derivative-bound-spherical}, together with the
two-sided $C^0$ estimate, all the normalized coefficients involving
$f$ and its first two derivatives are uniformly bounded.  More
precisely,
\begin{equation}\label{eq:normalized-profile-coefficient-bounds}
  \begin{gathered}
  0<c\leq
  \lambda^\beta f\left(\frac{\widetilde\rho}{\lambda}\right)
  \leq C,\\
  \lambda^{\beta-1}
  \left|f'\left(\frac{\widetilde\rho}{\lambda}\right)\right|
  +\lambda^{\beta-2}
  \left|f''\left(\frac{\widetilde\rho}{\lambda}\right)\right|
  \leq C.
  \end{gathered}
\end{equation}
Since
$\widetilde\Phi=\lambda^\beta
f(\widetilde\rho/\lambda)\sigma_k^\alpha(\hat h)$,
the lower speed bound in
Lemma~\ref{lem:normalized-speed-lower-bound-spherical} and
\eqref{eq:normalized-profile-coefficient-bounds} imply
\begin{equation}\label{eq:sigma-k-lower-bound-before-C2}
  \sigma_k(\hat h)\geq c>0.
\end{equation}
\begin{lemma}[Convexity preservation]
\label{lem:convexity-preservation-spherical}
Assume that the conditions of either Theorem~\ref{thm:main} or
Theorem~\ref{thm:main-critical} hold.  On the maximal time interval on
which the solution is strictly convex, define
\[
  \Lambda:=\lambda_{\max}(\check h)
  =\frac{1}{\hat\kappa_1}.
\]
Then there exists a constant $C>0$, depending only on the initial hypersurface, but independent of the endpoint
of this interval, such that
\begin{equation}\label{eq:lemma54-Lambda-bound-spherical}
  \Lambda\leq C.
\end{equation}
Consequently, strict convexity is preserved throughout the maximal
existence interval.
\end{lemma}
\begin{proof}
\medskip
\noindent\emph{Barrier reduction.}
Since $\check h$ is self-adjoint with respect to $\widetilde g$, its
largest eigenvalue is given by the Rayleigh quotient
\begin{equation}\label{eq:inverse-test-quantity}
  \Lambda(x,\tau):=\lambda_{\max}(\check h)=\sup_{0\neq v\in T_x M}
    \frac{\widetilde g(\check h v,v)}
         {|v|_{\widetilde g}^{2}}.
\end{equation}
On the maximal strictly convex interval $[0,\tau_*)$,
\begin{equation}\label{eq:Lambda-and-minimum-curvature}
  \Lambda=\frac{1}{\hat\kappa_1}<\infty.
\end{equation}
Since $\Lambda$ is only locally Lipschitz, we derive its differential
inequality in the barrier sense.  We first recall the corresponding
definition of a barrier subsolution.

\medskip
\noindent\emph{Barrier convention.}
  Let $J$ be a time interval.  A continuous function $u$ defined on
  $M\times J$ satisfies
  \[
    \mathcal Lu
    \leq\mathcal F(\widetilde\nabla u,u,x,\tau)
  \]
  in the barrier sense if, for every $(x_0,\tau_0)\in M\times J$,
  there exist a neighbourhood $U$ of $x_0$, a number $\epsilon>0$,
  and a smooth function
  \[
    \psi\in C^\infty
    \bigl(U\times(\tau_0-\epsilon,\tau_0]\bigr)
  \]
  such that $\psi\leq u$ on $U\times(\tau_0-\epsilon,\tau_0]$, with
  equality at $(x_0,\tau_0)$, and
  \begin{equation}\label{eq:barrier-subsolution-definition}
      \mathcal L\psi(x_0,\tau_0)
      \leq
      \mathcal F\bigl(
        \widetilde\nabla\psi(x_0,\tau_0),
        \psi(x_0,\tau_0),x_0,\tau_0
      \bigr).
  \end{equation}

Fix an arbitrary point $(x_0,\tau_0)$ and choose normal coordinates
$(x^1,\ldots,x^n)$ for $\widetilde g(\tau_0)$, centred at $x_0$.
Choose the first coordinate direction so that, at $x_0$, $\p_{x^1}$ is
a unit eigenvector of $\check h$ corresponding to
$\Lambda(x_0,\tau_0)$.
Regard the coordinate vector field $\p_{x^1}$ as time-independent,
and define the smooth function
\begin{equation}\label{eq:calabi-support-function}
  \phi(x,\tau)
  :=\frac{
    \widetilde g
    \bigl(\check h(\p_{x^1}),\p_{x^1}\bigr)}{
    \widetilde g(\p_{x^1},\p_{x^1})}.
\end{equation}
By the variational characterization of the largest eigenvalue,
\begin{equation}\label{eq:calabi-support-comparison}
  \phi(x,\tau)\leq\Lambda(x,\tau),
  \qquad
  \phi(x_0,\tau_0)=\Lambda(x_0,\tau_0).
\end{equation}
All the following derivative computations are evaluated at
$(x_0,\tau_0)$.  At this point,
\begin{equation}\label{eq:calabi-contact-data}
  \widetilde g(\p_{x^1},\p_{x^1})=1,
  \qquad
  \check h(\p_{x^1})=\Lambda\p_{x^1},
  \qquad
  \widetilde\nabla\p_{x^1}=0,
  \qquad
  \partial_\tau\p_{x^1}=0.
\end{equation}

For the time derivative,
\begin{equation}\label{eq:calabi-time-derivative-proof}
  \begin{aligned}
    \partial_\tau\phi
    &=(\partial_\tau\widetilde g)
      \bigl(\check h(\p_{x^1}),\p_{x^1}\bigr)
      +\widetilde g
      \bigl((\partial_\tau\check h)(\p_{x^1}),
        \p_{x^1}\bigr)
      -\Lambda(\partial_\tau\widetilde g)
      \bigl(\p_{x^1},\p_{x^1}\bigr)\\
    &=\widetilde g
      \bigl((\partial_\tau\check h)(\p_{x^1}),
        \p_{x^1}\bigr)\\
        & =(\partial_\tau\check h)_1{}^1.
  \end{aligned}
\end{equation}

For the spatial Hessian, metric compatibility and
\eqref{eq:calabi-contact-data} give
\begin{equation*}
  \begin{aligned}
    \widetilde\nabla_p\widetilde\nabla_q
      \Bigl[
        \widetilde g
        \bigl(\check h(\p_{x^1}),\p_{x^1}\bigr)
      \Bigr]=&\widetilde g\bigl(
       (\widetilde\nabla_p\widetilde\nabla_q\check h)
       (\p_{x^1}),\p_{x^1}\bigr)+\widetilde g\bigl(
       \check h(\widetilde\nabla_p\widetilde\nabla_q\p_{x^1}),
       \p_{x^1}\bigr)\\
     &+\widetilde g\bigl(
       \check h(\p_{x^1}),
       \widetilde\nabla_p\widetilde\nabla_q\p_{x^1}\bigr),
       \end{aligned}
       \end{equation*}
       and
\begin{equation*}     
    \widetilde\nabla_p\widetilde\nabla_q
      \Bigl[
        \widetilde g(\p_{x^1},\p_{x^1})
      \Bigr]=\widetilde g\bigl(
       \widetilde\nabla_p\widetilde\nabla_q\p_{x^1},
       \p_{x^1}\bigr)+\widetilde g\bigl(
       \p_{x^1},
       \widetilde\nabla_p\widetilde\nabla_q\p_{x^1}\bigr).
\end{equation*}
Since $\check h$ is diagonal at $(x_0,\tau_0)$, we have
\begin{equation}\label{eq:calabi-cross-term-identity}
  \widetilde g\bigl(
    \check h(\widetilde\nabla_p\widetilde\nabla_q\p_{x^1}),
    \p_{x^1}\bigr)=
    \widetilde g\bigl(
      \check h(\p_{x^1}),
      \widetilde\nabla_p\widetilde\nabla_q\p_{x^1}
    \bigr)=
    \Lambda\widetilde g\bigl(
      \widetilde\nabla_p\widetilde\nabla_q\p_{x^1},
      \p_{x^1}
    \bigr).
\end{equation}
Hence,
\begin{equation}\label{eq:calabi-spatial-hessian-conclusion}
  \begin{aligned}
    \widetilde\nabla_p\widetilde\nabla_q\phi
    ={}&\widetilde\nabla_p\widetilde\nabla_q
      \Bigl[
        \widetilde g
        \bigl(\check h(\p_{x^1}),\p_{x^1}\bigr)
      \Bigr]\\
    &-\Lambda\widetilde\nabla_p\widetilde\nabla_q
      \Bigl[
        \widetilde g(\p_{x^1},\p_{x^1})
      \Bigr]\\
    ={}&\widetilde g\bigl(
       (\widetilde\nabla_p\widetilde\nabla_q\check h)
       (\p_{x^1}),\p_{x^1}\bigr)
     =(\widetilde\nabla_p\widetilde\nabla_q\check h)_1{}^1.
  \end{aligned}
\end{equation}
We have therefore proved, at $(x_0,\tau_0)$, that
\begin{equation}\label{eq:calabi-L-relation}
  \mathcal L\phi
  =(\mathcal L\check h)_1{}^1.
\end{equation}
\medskip
\noindent\emph{The component evolution inequality.}
In the following display, $f$, $f'$, and $f''$ are evaluated at
$\widetilde\rho/\lambda$.  Taking $i=j=1$ in
\eqref{eq:full-inverse-weingarten-evolution-spherical} gives
\begin{equation}\label{eq:inverse-weingarten-11-component}
  \begin{aligned}
  \mathcal L\check h_1{}^1
  ={}&-\Lambda^2\lambda^\beta f
    \alpha\sigma_k^{\alpha-1}\ddot\sigma_k^{pq,rs}
    \widetilde\nabla_1\hat h_{pq}
    \widetilde\nabla_1\hat h_{rs}\\
  &-\Lambda^2\lambda^\beta f
    \alpha(\alpha-1)\sigma_k^{\alpha-2}
    \dot\sigma_k^{pq}\dot\sigma_k^{rs}
    \widetilde\nabla_1\hat h_{pq}
    \widetilde\nabla_1\hat h_{rs}\\
    &-2\Lambda^2\lambda^\beta f
    \alpha\sigma_k^{\alpha-1}\dot\sigma_k^{rs}
    (\widetilde\nabla_r\hat h_1{}^b)\check h_b{}^p
    (\widetilde\nabla_s\hat h_p{}^1)\\
  &-2\Lambda^2\lambda^{\beta-1}f'
    \alpha\sigma_k^{\alpha-1}\dot\sigma_k^{pq}
    \widetilde\nabla_1\widetilde\rho\,
    \widetilde\nabla_1\hat h_{pq}\\
  &-\Lambda^2\lambda^{\beta-1}f'\sigma_k^\alpha
    \widetilde\nabla_1\widetilde\nabla_1\widetilde\rho-\Lambda^2\lambda^{\beta-2}f''\sigma_k^\alpha
    |\widetilde\nabla_1\widetilde\rho|^2\\
  &-\lambda^\beta f\alpha\sigma_k^{\alpha-1}
    \left(
      \dot\sigma_k^{pq}(\hat h^2)_{pq}
      -\lambda^{-2}\dot\sigma_k^{pq}\widetilde g_{pq}
    \right)\Lambda\\
  &+\lambda^\beta f(k\alpha-1)\sigma_k^\alpha
    -\lambda^{\beta-2}f(1+k\alpha)\sigma_k^\alpha\Lambda^2
    +\gamma_0\Lambda.
  \end{aligned}
\end{equation}
By \eqref{eq:calabi-L-relation}, every upper estimate derived below for
the right-hand side of
\eqref{eq:inverse-weingarten-11-component} is an upper barrier estimate
for $\mathcal L\Lambda$.  We shall use this interpretation without
further comment.

\medskip
\noindent\emph{The inverse-concavity terms.}
Put
\begin{equation}\label{eq:F-and-G}
  G:=\sigma_k^{1/k}(\hat h),
\end{equation}
then the standard inverse-concavity of the operator $\sigma_k^{\frac{1}{k}}$ \cite{And07} yields
\begin{equation}\label{eq:inverse-concavity-inequality}
    \ddot G^{pq,rs}\eta_{pq}\eta_{rs}
    +2\dot G^{ij}\check h^{pq}\eta_{ip}\eta_{jq}
    \geq
    2G^{-1}\bigl(\dot G^{pq}\eta_{pq}\bigr)^2.
\end{equation}
Using $F=G^{k\alpha}$, we derive
\begin{equation}\label{eq:inverse-concavity-for-F}
  \begin{aligned}
  &\ddot F^{pq,rs}\eta_{pq}\eta_{rs}
   +2\dot F^{ij}\check h^{pq}\eta_{ip}\eta_{jq}\\
  ={}&k\alpha G^{k\alpha-1}
    \left(
      \ddot G^{pq,rs}\eta_{pq}\eta_{rs}
      +2\dot G^{ij}\check h^{pq}\eta_{ip}\eta_{jq}
    \right)
    +k\alpha(k\alpha-1)G^{k\alpha-2}
      \bigl(\dot G^{pq}\eta_{pq}\bigr)^2\\
  \geq{}&k\alpha(k\alpha+1)G^{k\alpha-2}
      \bigl(\dot G^{pq}\eta_{pq}\bigr)^2.
  \end{aligned}
\end{equation}
Set
$\eta_{pq}=\widetilde\nabla_1\hat h_{pq}$. Then the first three gradient terms in
\eqref{eq:inverse-weingarten-11-component} are exactly the negative
factor $-\Lambda^2\lambda^\beta f(\ddot F^{pq,rs}\eta_{pq}\eta_{rs}
   +2\dot F^{ij}\check h^{pq}\eta_{ip}\eta_{jq})$.  Since $f>0$ and
$\dot G^{pq}\eta_{pq}=\widetilde\nabla_1G$, we conclude that
\begin{equation}\label{eq:inverse-gradient-quadratic-combination}
  \begin{aligned}
  &-\Lambda^2
  \lambda^\beta f
  \Bigl(
    \ddot F^{pq,rs}\eta_{pq}\eta_{rs}
    +2\dot F^{ij}\check h^{pq}\eta_{ip}\eta_{jq}
  \Bigr)\\
  &\quad\leq
  -\Lambda^2
  \lambda^\beta f
  k\alpha(k\alpha+1)G^{k\alpha-2}
  |\widetilde\nabla_1G|^2.
  \end{aligned}
\end{equation}

\medskip
\noindent\emph{The $\widetilde\rho$-derivative terms.}

Here and below $f$, $f'$, and $f''$ are evaluated at
$\widetilde\rho/\lambda$.  The mixed radial term in
\eqref{eq:inverse-weingarten-11-component} becomes
\begin{equation}\label{eq:mixed-radial-gradient-term}
  -2\Lambda^2\lambda^{\beta-1}f'
  \widetilde\nabla_1\widetilde\rho\,
  \widetilde\nabla_1F
  =-2\Lambda^2\lambda^{\beta-1}f'
  k\alpha G^{k\alpha-1}
  \widetilde\nabla_1\widetilde\rho\,
  \widetilde\nabla_1G.
\end{equation}
Combining \eqref{eq:inverse-gradient-quadratic-combination} and
\eqref{eq:mixed-radial-gradient-term}, and completing the square, yields
\begin{equation}\label{eq:completed-square-gradient-bound}
  \begin{aligned}
  &-\Lambda^2k\alpha\lambda^\beta fG^{k\alpha-2}
  \Bigg[
    (k\alpha+1)|\widetilde\nabla_1G|^2
    +2\frac{f'}{\lambda f}G
      \widetilde\nabla_1\widetilde\rho\,
      \widetilde\nabla_1G
  \Bigg]\\
  &\quad\leq
  \Lambda^2\frac{k\alpha}{k\alpha+1}
  \lambda^{\beta-2}\frac{(f')^2}{f}
  F|\widetilde\nabla_1\widetilde\rho|^2.
  \end{aligned}
\end{equation}
Recall
\begin{equation}\label{eq:unscaled-radial-hessian}
  \nabla_i\nabla_j\rho
  =\cot\rho\,
    \bigl(
      g_{ij}-\nabla_i\rho\,\nabla_j\rho
    \bigr)
   -v^{-1}h_{ij}.
\end{equation}
Because $\lambda$ is spatially constant, $\widetilde\nabla=\nabla$ and
\begin{equation*}
  \widetilde\nabla_i\widetilde\nabla_j\widetilde\rho
  =\lambda\nabla_i\nabla_j\rho.
\end{equation*}
Using
\begin{equation*}
  \rho=\frac{\widetilde\rho}{\lambda},
  \qquad
  g_{ij}=\lambda^{-2}\widetilde g_{ij},
  \qquad
  \nabla_i\rho
    =\lambda^{-1}\widetilde\nabla_i\widetilde\rho,
  \qquad
  h_{ij}=\lambda^{-1}\hat h_{ij},
\end{equation*}
we obtain from \eqref{eq:unscaled-radial-hessian}
\begin{equation}\label{eq:scaled-radial-hessian}
  \widetilde\nabla_i\widetilde\nabla_j\widetilde\rho
  =\lambda^{-1}\cot
    \left(\frac{\widetilde\rho}{\lambda}\right)
    \left(
      \widetilde g_{ij}
      -\widetilde\nabla_i\widetilde\rho\,
       \widetilde\nabla_j\widetilde\rho
    \right)
   -v^{-1}\hat h_{ij}.
\end{equation}
Consequently, the two purely radial second-order terms in
\eqref{eq:inverse-weingarten-11-component} equal
\begin{equation}\label{eq:pure-radial-terms-at-maximum}
  \begin{aligned}
  &-\Lambda^2F\Bigg[
   \lambda^{\beta-2}f''
     |\widetilde\nabla_1\widetilde\rho|^2
   +\lambda^{\beta-2}f'\cot
     \left(\frac{\widetilde\rho}{\lambda}\right)
     \left(1-|\widetilde\nabla_1\widetilde\rho|^2\right)
  \Bigg]
  +\Lambda\lambda^{\beta-1}f'v^{-1}F.
  \end{aligned}
\end{equation}
Combining \eqref{eq:completed-square-gradient-bound} and
\eqref{eq:pure-radial-terms-at-maximum}, we find that all the terms
involving derivatives of $\widetilde\rho$, together with the negative
term retained from \eqref{eq:inverse-gradient-quadratic-combination}, are
bounded above by
\begin{equation}\label{eq:all-radial-derivative-terms-bound}
  \begin{aligned}
  &-\Lambda^2F\Bigg[
    \lambda^{\beta-2}f'
      \cot\left(\frac{\widetilde\rho}{\lambda}\right)
      \left(1-|\widetilde\nabla_1\widetilde\rho|^2\right)\\
  &\hspace{23mm}
    +\lambda^{\beta-2}
      \left(
        f''-\frac{k\alpha}{k\alpha+1}\frac{(f')^2}{f}
      \right)
      |\widetilde\nabla_1\widetilde\rho|^2
  \Bigg]
  +C\widetilde\Phi\Lambda,
  \end{aligned}
\end{equation}
where we used the fact that
\[
  \Lambda\lambda^{\beta-1}f'v^{-1}F
  =\Lambda\frac{f'}{\lambda f}v^{-1}\widetilde\Phi
  \leq C\widetilde\Phi\Lambda.
\]

\medskip
\noindent\emph{The zero-order curvature terms.}

It remains to control the zero-order curvature terms.  At $i=j=1$ they
are
\begin{equation}\label{eq:inverse-zero-order-terms}
  \begin{aligned}
  &-\lambda^\beta f\,
  \alpha\sigma_k^{\alpha-1}
  \left(
    \dot\sigma_k^{pq}(\hat h^2)_{pq}
    -\lambda^{-2}\dot\sigma_k^{pq}\widetilde g_{pq}
  \right)\Lambda\\
  &\quad+\lambda^\beta f(k\alpha-1)F
  -\lambda^{\beta-2}f(1+k\alpha)F\Lambda^2
  +\gamma_0\Lambda.
  \end{aligned}
\end{equation}
The first term in the first line is nonpositive.  Since every principal
curvature is positive on $[0,\tau_*)$, Euler's identity gives
\begin{equation}\label{eq:spherical-dot-F-trace-bound}
  k\alpha F
  =\alpha\sigma_k^{\alpha-1}
    \dot\sigma_k^{pq}\hat h_{pq}
  \geq
  \Lambda^{-1}\alpha\sigma_k^{\alpha-1}
    \dot\sigma_k^{pq}\widetilde g_{pq}.
\end{equation}
Consequently, the two ambient-curvature terms satisfy
\begin{equation}\label{eq:spherical-ambient-curvature-good-sign}
  \begin{aligned}
  &\lambda^{\beta-2}f\alpha\sigma_k^{\alpha-1}
    \dot\sigma_k^{pq}\widetilde g_{pq}\Lambda
  -\lambda^{\beta-2}f(1+k\alpha)F\Lambda^2\\
  &\qquad\leq-\lambda^{\beta-2}fF\Lambda^2\leq0.
  \end{aligned}
\end{equation}
Therefore,
\begin{equation}\label{eq:inverse-zero-order-bound}
  \text{the expression in \eqref{eq:inverse-zero-order-terms}}
  \leq C\widetilde\Phi+C\Lambda.
\end{equation}

Combining \eqref{eq:all-radial-derivative-terms-bound} and
\eqref{eq:inverse-zero-order-bound}, the maximum-point inequality becomes
\begin{equation}\label{eq:Lambda-key-coefficient}
  \begin{aligned}
  \mathcal L\Lambda
  \leq{}&C\widetilde\Phi\Lambda+C\widetilde\Phi+C\Lambda\\
  &-\Lambda^2F\Bigg[
   \lambda^{\beta-2}f'
     \cot\left(\frac{\widetilde\rho}{\lambda}\right)
     \left(1-|\widetilde\nabla_1\widetilde\rho|^2\right)\\
  &\hspace{20mm}
   +\lambda^{\beta-2}
    \left(
      f''-\frac{k\alpha}{k\alpha+1}\frac{(f')^2}{f}
    \right)
    |\widetilde\nabla_1\widetilde\rho|^2
  \Bigg].
  \end{aligned}
\end{equation}

The gradient estimate gives
\begin{equation}\label{eq:radial-gradient-good-factor}
  1-|\widetilde\nabla_1\widetilde\rho|^2
  \geq1-|\widetilde\nabla^M\widetilde\rho|^2
  =v^{-2}\geq c>0.
\end{equation}

\medskip
\noindent\emph{Coercivity and the maximum principle.}
Choose $R\in(0,\pi/2)$ so that
$0<r=\widetilde\rho/\lambda\leq R$ on $[0,\tau_*)$.  Lemmas
\ref{lem:model-profile-comparison-spherical},
\ref{lem:profile-logarithmic-derivative-bound-spherical}, and
\ref{lem:profile-second-derivative-bound-spherical} imply that, for
every $r\in(0,R]$,
\begin{equation}\label{eq:general-f-structural-growth}
  \begin{aligned}
  c_f r^\beta&\leq f(r)\leq C_f r^\beta,\\
  c_f r^{\beta-1}&\leq f'(r)\leq C_f r^{\beta-1},\\
  |f''(r)|&\leq C_f r^{\beta-2}.
  \end{aligned}
\end{equation}
The root-convexity assumption in the corresponding theorem gives
\begin{equation}\label{eq:general-f-root-convexity}
  \left(f^{\frac{1}{1+k\alpha}}\right)''\geq0.
\end{equation}
The latter condition is equivalent to
\begin{equation}\label{eq:general-f-inverse-concavity-combination}
  f''-\frac{k\alpha}{1+k\alpha}\frac{(f')^2}{f}
  =(1+k\alpha)f^{\frac{k\alpha}{1+k\alpha}}
    \left(f^{\frac{1}{1+k\alpha}}\right)''
  \geq0.
\end{equation}

Set $q=|\widetilde\nabla_1\widetilde\rho|^2$.  By
\eqref{eq:radial-gradient-good-factor}, there is a fixed
$\vartheta>0$ such that $1-q\geq\vartheta$.  Therefore the bracket in
\eqref{eq:Lambda-key-coefficient} satisfies
\begin{equation}\label{eq:general-f-uniform-bracket}
  \begin{aligned}
  &\lambda^{\beta-2}\left[
    f'(r)\cot r\,(1-q)
    +\left(
      f''-\frac{k\alpha}{1+k\alpha}\frac{(f')^2}{f}
     \right)q
  \right]\\
  &\quad\geq
  \vartheta\lambda^{\beta-2}f'(r)\cot r\\
  &\quad\geq
  c\lambda^{\beta-2}r^{\beta-1}\cot r\\
  &\quad\geq
  c\lambda^{\beta-2}r^{\beta-2}
  =c\widetilde\rho^{\beta-2}
  \geq c.
  \end{aligned}
\end{equation}
Here we used the positive lower bound
$r\cot r\geq R\cot R>0$ on $(0,R]$ and the two-sided $C^0$ bound for
$\widetilde\rho$.  Since
\[
  \widetilde\Phi=\lambda^\beta f(r)F
  \quad\text{and}\quad
  \lambda^\beta f(r)\leq C,
\]
the negative quadratic term in
\eqref{eq:Lambda-key-coefficient} is bounded above by
$-c\widetilde\Phi\Lambda^2$.  We therefore obtain
\begin{equation}\label{eq:Lambda-general-profile-all-time-inequality}
  \mathcal L\Lambda
  \leq-c\widetilde\Phi\Lambda^2
       +C\widetilde\Phi\Lambda
       +C\widetilde\Phi+C\Lambda.
\end{equation}
This is the point at which the positive lower speed bound, rather than
an upper speed bound, is used.  By
\eqref{eq:normalized-speed-lower-bound-spherical},
\[
  C\Lambda
  \leq\frac{C}{c_{\widetilde\Phi}}
       \widetilde\Phi\Lambda.
\]
After changing $C$, \eqref{eq:Lambda-general-profile-all-time-inequality}
becomes
\begin{equation}\label{eq:Lambda-factorized-by-speed-spherical}
  \mathcal L\Lambda
  \leq\widetilde\Phi
       \left(-c\Lambda^2+C\Lambda+C\right).
\end{equation}
Thus the right-hand side is negative whenever $\Lambda$ exceeds a
fixed constant; no upper bound for $\widetilde\Phi$ is needed in this
argument.  The maximum principle now gives
\begin{equation}\label{eq:minimum-curvature-bound}
  \Lambda\leq C_*,
  \qquad
  \hat\kappa_i\geq C_*^{-1}
  \qquad\text{on }\mathbb S^n\times[0,\tau_*),
\end{equation}
where $C_*$ is independent of $\tau_*$ and of the
strict-convexity margin.

The bound in \eqref{eq:minimum-curvature-bound} closes the
convexity continuation argument directly.  Indeed, if \(\tau_*\) were
a finite endpoint at which strict convexity were lost, then
\(\hat\kappa_1\to0\) and hence \(\Lambda\to\infty\), contradicting
\eqref{eq:minimum-curvature-bound}.  Therefore \(\tau_*\) coincides
with the endpoint of the maximal existence interval, and strict
convexity is preserved throughout that interval.
\end{proof}

\begin{lemma}[The normalized $C^2$ estimate]
\label{lem:normalized-C2-estimate-spherical}
Assume that the conditions of either Theorem~\ref{thm:main} or
Theorem~\ref{thm:main-critical} hold. Then there exists a constant
$C>0$, depending only on the initial hypersurface, such that
\begin{equation}\label{eq:lemma54-full-C2-conclusions-spherical}
  \widetilde\Phi\leq C,
  \qquad
  C^{-1}\leq\hat\kappa_i\leq C,
  \qquad
  |\nabla_\sigma^2\widetilde\rho|_\sigma\leq C.
\end{equation}
These estimates hold throughout the maximal existence interval.
\end{lemma}

\begin{proof}
By Lemma~\ref{lem:convexity-preservation-spherical}, the solution
remains strictly convex and
\begin{equation}\label{eq:C2-minimum-curvature-bound-spherical}
  \hat\kappa_i\geq c>0.
\end{equation}
It remains first to derive the upper bound for the normalized speed.
The $C^1$ estimate gives two-sided positive bounds for the normalized
support function $\widetilde u$.  Set
\begin{equation}\label{eq:C2-speed-test-function-spherical}
  b:=\frac12\inf_{\mathbb S^n\times[0,T)}\widetilde u,
  \qquad
  z:=\widetilde u-b,
  \qquad
  Q:=\log\widetilde\Phi-\log z,
\end{equation}
where $[0,T)$ is the maximal existence interval.  Thus $z$ is bounded
above and below by positive constants.

At a spatial maximum of $Q$, one has
$\widetilde\nabla\log\widetilde\Phi=\widetilde\nabla\log z$;
hence the gradient-square terms in $\mathcal LQ$ cancel.  Using
\eqref{eq:basic-normalized-speed-evolution-spherical} and
\eqref{eq:basic-normalized-support-evolution-spherical}, we obtain
\begin{equation}\label{eq:C2-speed-quotient-evolution-spherical}
  \begin{aligned}
  \mathcal LQ
  ={}& (\beta-k\alpha)\gamma_0
       -\gamma_0\frac{\widetilde u}{z}
       -\frac{f'(\rho)}{\lambda f(\rho)v}\widetilde\Phi
       +\frac{(1+k\alpha)\cos\rho}{z}\widetilde\Phi\\
    &-\frac{b}{z}\lambda^\beta f(\rho)
       \dot F^{ij}(\hat h^2)_{ij}
       +\lambda^{\beta-2}f(\rho)
       \dot F^{ij}\widetilde g_{ij}\\
    &-\frac{\lambda^\beta f'(\rho)F\sin\rho}{z}
       (1-v^{-2}).
  \end{aligned}
\end{equation}
Since $f'>0$ and $v\geq1$, the last term and the third term on the
right-hand side are nonpositive.  The remaining bounded coefficients
therefore give
\begin{equation}\label{eq:C2-speed-quotient-upper-bound-spherical}
  \mathcal LQ
  \leq C+C\widetilde\Phi
       -\frac{b}{z}\lambda^\beta f(\rho)
        \dot F^{ij}(\hat h^2)_{ij}
       +\lambda^{\beta-2}f(\rho)
        \dot F^{ij}\widetilde g_{ij}.
\end{equation}

The Newton--Maclaurin inequalities, the homogeneity of
$F=\sigma_k^\alpha$, and
\eqref{eq:normalized-profile-coefficient-bounds} imply
\begin{equation}\label{eq:C2-speed-coercive-term-spherical}
  \lambda^\beta f(\rho)
  \dot F^{ij}(\hat h^2)_{ij}
  \geq c\widetilde\Phi^{1+\frac1{k\alpha}}.
\end{equation}
On the other hand, Euler's identity and
\eqref{eq:C2-minimum-curvature-bound-spherical} give
\[
  k\alpha F=\dot F^{ij}\hat h_{ij}
  \geq c\dot F^{ij}\widetilde g_{ij},
\]
and hence
\begin{equation}\label{eq:C2-spherical-trace-term-bound}
  \lambda^{\beta-2}f(\rho)
  \dot F^{ij}\widetilde g_{ij}
  \leq C\lambda^{-2}\widetilde\Phi
  \leq C\widetilde\Phi.
\end{equation}
It follows from
\eqref{eq:C2-speed-quotient-upper-bound-spherical}--
\eqref{eq:C2-spherical-trace-term-bound} that, at a spatial maximum of
$Q$,
\[
  \mathcal LQ
  \leq C+C\widetilde\Phi
       -c\widetilde\Phi^{1+\frac1{k\alpha}}.
\]
The maximum principle and the two-sided bounds for $z$ yield
\begin{equation}\label{eq:C2-normalized-speed-upper-bound-spherical}
  \widetilde\Phi\leq C.
\end{equation}

Combining this estimate with the lower speed bound and
\eqref{eq:normalized-profile-coefficient-bounds}, we obtain
\begin{equation}\label{eq:C2-sigma-k-two-sided-spherical}
  0<c\leq\sigma_k(\hat h)\leq C.
\end{equation}
Because every $\hat\kappa_i$ is bounded below by a positive constant,
the upper bound for $\sigma_k(\hat h)$ bounds each $\hat\kappa_i$ from
above.  Therefore
\begin{equation}\label{eq:C2-principal-curvature-two-sided-spherical}
  C^{-1}\leq\hat\kappa_i\leq C.
\end{equation}
Finally, we estimate the Hessian on the parameter sphere
$(\mathbb S^n,\sigma)$.  In \eqref{eq:weingarten-map-rho}, the terms
containing the spherical Hessian of $\rho$ can be written as
\[
  -\left(
    \sigma^{jp}-\frac{\rho^j\rho^p}{\sin^2\rho\,v^2}
   \right)\rho_{pi}.
\]
Since $\rho=\widetilde\rho/\lambda$ and
$\hat h_i{}^j=\lambda^{-1}h_i{}^j$, equation
\eqref{eq:weingarten-map-rho} is therefore equivalent to
\begin{equation}\label{eq:C2-spherical-Hessian-from-Weingarten}
  \begin{aligned}
  \left(
    \sigma^{jp}
    -\frac{\widetilde\rho^j\widetilde\rho^p}
      {\lambda^2\sin^2\rho\,v^2}
  \right)\widetilde\rho_{pi}
  ={}&\lambda\sin\rho\cos\rho\,\delta_i{}^j\\
  &+\frac{\cot\rho}{\lambda v^2}
      \widetilde\rho_i\widetilde\rho^j
    -\lambda^2v\sin^2\rho\,\hat h_i{}^j.
  \end{aligned}
\end{equation}
Here all derivatives are taken with respect to $\sigma$.  The
coefficient matrix on the left-hand side has eigenvalues $1$ in the
directions orthogonal to $\nabla\widetilde\rho$ and $v^{-2}$ in its
gradient direction.  It and its inverse are therefore uniformly
bounded by the $C^1$ estimate.  Moreover, the $C^0$ estimate gives
two-sided bounds for $\lambda\sin\rho$ and a bound for
$\lambda^{-1}\cot\rho$, while
\eqref{eq:C2-principal-curvature-two-sided-spherical} bounds
$\hat h_i{}^j$.  Every term on the right-hand side of
\eqref{eq:C2-spherical-Hessian-from-Weingarten} is consequently
uniformly bounded.  Hence
\[
  |\nabla_\sigma^2\widetilde\rho|_\sigma\leq C.
\]
This proves \eqref{eq:lemma54-full-C2-conclusions-spherical}.
\end{proof}

\section{Exponential decay}

\begin{lemma}[Exponential decay of the normalized gradient]
\label{lem:exponential-gradient-decay-spherical}
Assume that the conditions of either Theorem~\ref{thm:main} or
Theorem~\ref{thm:main-critical} hold.  Then there exist constants
$C,c>0$ such that
\begin{equation}\label{eq:exponential-gradient-decay-spherical}
  |\nabla^\sigma\widetilde\rho(\cdot,\tau)|_\sigma
  \leq Ce^{-c\tau}
  \qquad\text{for all }\tau\geq0.
\end{equation}
\end{lemma}

\begin{proof}
We derive the gradient estimate by the maximum principle.  Define
\begin{equation}\label{eq:spherical-logarithmic-radius}
  \varphi(\theta,\tau)
  :=\log\tan(\frac{\rho(\theta,t(\tau))}2),
\end{equation}
here $\rho$ is the unnormalized radial function evaluated at $t(\tau)$. Since $\varphi_\rho=(\sin\rho)^{-1}$, it follows
\begin{equation}\label{eq:varphi-gradient-relations}
  \nabla\varphi=\frac{\nabla\rho}{\sin\rho},
  \qquad
  v=\sqrt{1+|\nabla\varphi|^2},
  \qquad
  \nabla\widetilde\rho=\lambda\sin\rho\,\nabla\varphi,
\end{equation}
and
\[
  \rho_i=\sin\rho\,\varphi_i,
  \qquad
  \rho_{ij}
  =\sin\rho\,\varphi_{ij}
   +\sin\rho\cos\rho\,\varphi_i\varphi_j.
\]
Substituting these identities into~\eqref{eq:weingarten-map-rho} and using~\eqref{eq:normalized-weingarten-map-definition}, we obtain
\begin{equation}\label{eq:hat-h-in-varphi-spherical}
  \hat h_i{}^j
  =\frac{1}{\lambda v\sin\rho}
    \left[
      \cos\rho\,\delta_i{}^j
      -\left(
         \sigma^{jp}-\frac{\varphi^j\varphi^p}{v^2}
       \right)\varphi_{pi}
    \right].
\end{equation}
Replacing $\rho$ in \eqref{eq:radial-flow} by
$\varphi$ and using \eqref{eq:time-and-lambda-identities}, we obtain
\begin{equation}\label{eq:varphi-evolution-spherical}
  \partial_\tau\varphi
  =-\lambda^{\beta-1}\frac{f(\rho)}{\sin\rho}
      \sigma_k^\alpha(\hat h)v.
\end{equation}

Put
\begin{equation}\label{eq:gradient-auxiliary-function-spherical}
  W:=\frac12|\nabla\varphi|^2,
\end{equation}
where the norm and derivatives are taken with respect to the standard
metric on $\mathbb S^n$. At a spatial maximum $(\theta_0,\tau)$ of $W$, we have
\begin{equation}\label{eq:C1-maximum-point-relation}
  \varphi^i\varphi_{ij}=0,\quad
  \nabla v=0,\quad
  \text{ and }
  \nabla^2W\leq0.
\end{equation}
Then equation \eqref{eq:varphi-evolution-spherical} gives
\begin{equation}\label{eq:W-time-derivative-first-expansion-spherical}
  \begin{aligned}
  \partial_\tau W
  ={}&-\varphi^m\nabla_m
      \left[
        \lambda^{\beta-1}\frac{f(\rho)}{\sin\rho}
        \sigma_k^\alpha(\hat h)v
      \right]\\
  ={}&-v\sigma_k^\alpha(\hat h)\,\varphi^m\nabla_m
      \left(\lambda^{\beta-1}\frac{f(\rho)}{\sin\rho}\right)-\lambda^{\beta-1}\frac{f(\rho)}{\sin\rho}v
       \frac{\partial\sigma_k^\alpha}
            {\partial\hat h_i{}^j}
       \varphi^m\nabla_m\hat h_i{}^j
         \end{aligned}
\end{equation}
where we used the fact that $\n v = 0$ at the maximum point.

Since $\rho_m=\sin\rho\,\varphi_m$, we have
\begin{equation}\label{eq:coefficient-and-v-gradient-at-W-maximum-spherical}
  \varphi^m\nabla_m
  \left(\lambda^{\beta-1}\frac{f(\rho)}{\sin\rho}\right)
  =\lambda^{\beta-1}\bigl(f'(\rho)-f(\rho)\cot\rho\bigr)
     |\nabla\varphi|^2.
\end{equation}

Differentiating \eqref{eq:hat-h-in-varphi-spherical} and using~\eqref{eq:C1-maximum-point-relation}, we obtain 
\begin{equation}\label{eq:directional-derivative-hat-h-spherical}
  \begin{aligned}
  \varphi^m\nabla_m\hat h_i{}^j
  ={}&-\cos\rho\,|\nabla\varphi|^2\hat h_i{}^j
      -\frac{\sin\rho}{\lambda v}
       |\nabla\varphi|^2\delta_i{}^j\\
     &-\frac{1}{\lambda v\sin\rho}
       \left(
         \sigma^{jp}-\frac{\varphi^j\varphi^p}{v^2}
       \right)\varphi^m\varphi_{pi m}.
  \end{aligned}
\end{equation}
Using the $k\alpha$-homogeneity identity
\begin{equation}\label{eq:sigma-k-alpha-Euler-identity-C1-spherical}
  \frac{\partial\sigma_k^\alpha}{\partial\hat h_i{}^j}
  \hat h_i{}^j
  =k\alpha\sigma_k^\alpha(\hat h),
\end{equation}
and substituting
\eqref{eq:coefficient-and-v-gradient-at-W-maximum-spherical} and
\eqref{eq:directional-derivative-hat-h-spherical} into
\eqref{eq:W-time-derivative-first-expansion-spherical}, we obtain
\begin{equation}\label{eq:W-before-third-derivative-commutation-spherical}
  \begin{aligned}
  \partial_\tau W
  ={}&\lambda^{\beta-2}f(\rho)
      (\dot\sigma_k^\alpha)^{pi}\varphi^m\varphi_{pi m}\\
   &
   -\lambda^{\beta-1}v\sigma_k^\alpha(\hat h)
      \left[f'(\rho)-(1+k\alpha)f(\rho)\cot\rho\right]
      |\nabla\varphi|^2\\
   &+\lambda^{\beta-2}f(\rho)
      \frac{\partial\sigma_k^\alpha}
           {\partial\hat h_i{}^j}\delta_i{}^j
      |\nabla\varphi|^2.
  \end{aligned}
\end{equation}
where
\begin{equation}\label{eq:C1-linearized-coefficient-spherical}
  (\dot\sigma_k^\alpha)^{pq}
  :=\frac{\partial\sigma_k^\alpha}{\partial\hat h_q{}^j}g^{jp} = \frac{\partial\sigma_k^\alpha}{\partial\hat h_q{}^j}
  \frac1{\sin^2\rho}
  \left(
    \sigma^{jp}-\frac{\varphi^j\varphi^p}{v^2}
  \right).
\end{equation}
Clearly, $(\dot\sigma_k^\alpha)^{pq}$ is positive definite.

By the Ricci identity on $\mathbb{S}^n$, we have
\begin{equation}\label{eq:third-derivative-commutation-for-W-spherical}
  \varphi^m\varphi_{pi m}
  =\nabla_p\nabla_iW
   -\varphi_p{}^m\varphi_{mi}
   -\bigl(\sigma_{pi}|\nabla\varphi|^2-\varphi_p\varphi_i\bigr).
\end{equation}
Consequently, at a maximum point of $W$,
\begin{equation}\label{eq:gradient-maximum-calculation-spherical}
  \begin{aligned}
  \partial_\tau W
  ={}&\lambda^{\beta-2}f(\rho)
      (\dot\sigma_k^\alpha)^{pi}\nabla_p\nabla_iW-\lambda^{\beta-2}f(\rho)
      (\dot\sigma_k^\alpha)^{pi}
      \varphi_p{}^m\varphi_{mi}\\
   &-\lambda^{\beta-2}f(\rho)
      (\dot\sigma_k^\alpha)^{pi}
      \bigl(\sigma_{pi}|\nabla\varphi|^2-\varphi_p\varphi_i\bigr)\\
   &-\lambda^{\beta-1}v\sigma_k^\alpha(\hat h)
      \left[f'(\rho)-(1+k\alpha)f(\rho)\cot\rho\right]
      |\nabla\varphi|^2\\
   &+\lambda^{\beta-2}f(\rho)
      \frac{\partial\sigma_k^\alpha}
           {\partial\hat h_i{}^j}\delta_i{}^j
      |\nabla\varphi|^2.
  \end{aligned}
\end{equation}

To display the sign of every term, choose at the maximum point an
orthonormal frame such that
\begin{equation}\label{eq:C1-adapted-frame-spherical}
  \varphi_1=|\nabla\varphi|=\sqrt{2W},
  \qquad \varphi_a=0\quad(a=2,\ldots,n).
\end{equation}
If $W>0$, the relation $\varphi^i\varphi_{ij}=0$ at the maximum
point implies $\varphi_{1j}=0$ for every $j$.  We may therefore rotate
only the orthogonal complement of $\partial_{x^1}$ so that the
symmetric block $(\varphi_{ab})_{a,b\geq2}$ is diagonal.  This rotation
does not change the direction of $\nabla\varphi$.  In the resulting
frame,
\[
  \left(
    \sigma^{jp}-\frac{\varphi^j\varphi^p}{v^2}
  \right)
  =\operatorname{diag}(v^{-2},1,\ldots,1).
\]
It now follows from \eqref{eq:hat-h-in-varphi-spherical} that
$\hat h_i{}^j$ is diagonal in the same frame.  Thus
$\partial_{x^1}$ is automatically a principal direction, and
\begin{equation}\label{eq:principal-curvatures-at-W-maximum-spherical}
  \hat\kappa_1=\frac{\cot\rho}{\lambda v},
  \qquad
  \hat\kappa_a
  =\frac{\cos\rho-\varphi_{aa}}
         {\lambda v\sin\rho},
  \quad a=2,\ldots,n.
\end{equation}
Then,
$$\sigma_{pi}(\dot{\sigma}_k^{\alpha})^{pi} = \alpha\sigma_{k}^{\alpha-1}\frac{\p \sigma_k}{\p \hat{h}_i{ }^j}\delta_j{ }^i\frac{1}{\sin^2\rho} - \alpha\sigma_k^{\alpha-1}\frac{\p\sigma_k}{\p \hat{h}_1{ }^1}\frac{1}{\sin^2\rho}\frac{2W}{v^2}$$
and
$$(\dot{\sigma}_k^{\alpha})^{pi}\varphi_p\varphi_i = \alpha\sigma_{k}^{\alpha-1}\frac{\p \sigma_k}{\p \hat{h}_1{ }^1}\frac{1}{\sin^2\rho}2W - \alpha\sigma_k^{\alpha-1}\frac{\p\sigma_k}{\p \hat{h}_1{ }^1}\frac{1}{\sin^2\rho}\frac{4W^2}{v^2}.$$
Combining these two identities and using
$|\nabla\varphi|^2=2W$, we obtain
\begin{equation}\label{eq:spherical-curvature-gradient-contraction}
  \begin{aligned}
  &\lambda^{\beta-2}f(\rho)
   (\dot\sigma_k^\alpha)^{pi}
   \bigl(\sigma_{pi}|\nabla\varphi|^2
         -\varphi_p\varphi_i\bigr)\\
  ={}&\frac{2W\alpha\lambda^{\beta-2}f(\rho)
             \sigma_k^{\alpha-1}}{\sin^2\rho}
       \left(
         \frac{\partial\sigma_k}{\partial\hat h_i{}^j}
           \delta_i{}^j
         -\frac{\partial\sigma_k}{\partial\hat h_1{}^1}
       \right)\\
  ={}&\frac{2W\alpha\lambda^{\beta-2}f(\rho)
             \sigma_k^{\alpha-1}}{\sin^2\rho}
       \sum_{a=2}^n
         \frac{\partial\sigma_k}{\partial\hat\kappa_a}.
  \end{aligned}
\end{equation}
Consequently, the first term in
\eqref{eq:gradient-maximum-calculation-spherical} is nonpositive, and
the remaining terms give
\begin{equation}\label{eq:gradient-maximum-component-expansion-spherical}
  \begin{aligned}
  D^+W_{\max}
  \leq{}&-\lambda^{\beta-1}v\sigma_k^\alpha(\hat h)
      \left[f'(\rho)-(1+k\alpha)f(\rho)\cot\rho\right]
      2W_{\max}\\
   &-\lambda^{\beta-2}\frac{f(\rho)}{\sin^2\rho}
      \sum_{a=2}^n
      \frac{\partial\sigma_k^\alpha}{\partial\hat\kappa_a}
      \varphi_{aa}^2\\
   &+\lambda^{\beta-2}f(\rho)2W_{\max}
      \left(
        \frac{\partial\sigma_k^\alpha}{\partial\hat\kappa_1}
        -\cot^2\rho\sum_{a=2}^n
         \frac{\partial\sigma_k^\alpha}{\partial\hat\kappa_a}
      \right).
  \end{aligned}
\end{equation}

For completeness, the sign of the profile term follows directly from
\eqref{eq:main-profile-root-convexity} or
\eqref{eq:critical-profile-root-convexity}.  Indeed, convexity of
$f^{1/(1+k\alpha)}$, together with $f(0)=0$, gives
\begin{equation}\label{eq:profile-first-derivative-sign-spherical}
  rf'(r)\geq(1+k\alpha)f(r),
  \qquad
  f'(r)-(1+k\alpha)f(r)\cot r\geq0.
\end{equation}

We now use the normalized $C^2$ estimate.  By
Lemma~\ref{lem:normalized-C2-estimate-spherical}, the principal
curvature vector $\hat\kappa$ remains in a fixed compact subset of the
positive cone.  Hence
\begin{equation}\label{eq:uniform-ellipticity-after-C2-spherical}
  0<c\leq
  \frac{\partial\sigma_k^\alpha}{\partial\hat\kappa_i}
  \leq C,
  \qquad
  0<c\leq\sigma_k^\alpha(\hat h)\leq C.
\end{equation}
Moreover, the $C^0$ estimate and
\eqref{eq:normalized-profile-coefficient-bounds} give
\begin{equation}\label{eq:decay-coefficients-spherical}
  0<c\leq\lambda^\beta f(\rho)\leq C,
  \qquad
  0<c\leq\lambda^{-1}\cot\rho\leq C.
\end{equation}

Suppose first that $n\geq2$.  Using
\eqref{eq:uniform-ellipticity-after-C2-spherical} and
\eqref{eq:decay-coefficients-spherical}, the last line of
\eqref{eq:gradient-maximum-component-expansion-spherical} satisfies
\begin{equation}\label{eq:spherical-ambient-gradient-coercivity}
  \begin{aligned}
  &2W_{\max}\lambda^{\beta-2}f(\rho)
      \left(
        \frac{\partial\sigma_k^\alpha}{\partial\hat\kappa_1}
        -\cot^2\rho\sum_{a=2}^n
         \frac{\partial\sigma_k^\alpha}{\partial\hat\kappa_a}
      \right)\\
  ={}&2W_{\max}\lambda^\beta f(\rho)
      \left(
        \lambda^{-2}
        \frac{\partial\sigma_k^\alpha}{\partial\hat\kappa_1}
        -(\lambda^{-1}\cot\rho)^2
          \sum_{a=2}^n
          \frac{\partial\sigma_k^\alpha}{\partial\hat\kappa_a}
      \right)\\
  \leq{}&-cW_{\max}
  \end{aligned}
\end{equation}
for all sufficiently large $\tau$.  The first two terms on the
right-hand side of
\eqref{eq:gradient-maximum-component-expansion-spherical} are
nonpositive by \eqref{eq:profile-first-derivative-sign-spherical} and
ellipticity.  We consequently obtain
\begin{equation}\label{eq:W-exponential-inequality-n-geq-two}
  D^+W_{\max}\leq-cW_{\max}
\end{equation}
for all sufficiently large $\tau$.

It remains to consider $n=1$.  This case occurs only under the
supercritical assumptions of Theorem~\ref{thm:main}.  Put
$\mu:=\beta-1-k\alpha>0$.  The expansion
$f(r)=r^\beta+o(r^\beta)$ and the corresponding differentiated
expansion give
\begin{equation}\label{eq:supercritical-profile-gradient-coercivity}
  f'(r)-(1+k\alpha)f(r)\cot r
  =\bigl(\mu+o(1)\bigr)r^{\beta-1}
  \geq cr^{\beta-1}
\end{equation}
for sufficiently small $r$.  Since
$\lambda^{\beta-1}\rho^{\beta-1}
=\widetilde\rho^{\beta-1}$, the first term in
\eqref{eq:gradient-maximum-component-expansion-spherical} is bounded
above by $-cW_{\max}$.  When $n=1$, the last line of that equation has
no negative summation term, but its remaining positive part is bounded
by $C\lambda^{-2}W_{\max}$.  After increasing the initial time if
necessary, it is absorbed by the profile term.  Thus
\begin{equation}\label{eq:W-exponential-inequality-n-one}
  D^+W_{\max}\leq-cW_{\max}
\end{equation}
also in this case.

Combining \eqref{eq:W-exponential-inequality-n-geq-two} and
\eqref{eq:W-exponential-inequality-n-one}, and enlarging the constant
to cover the remaining compact time interval, yields
\begin{equation}\label{eq:W-exponential-decay-spherical}
  W_{\max}(\tau)\leq Ce^{-c\tau}.
\end{equation}
Finally, \eqref{eq:varphi-gradient-relations} and the upper bound for
$\lambda\sin\rho$ imply
\[
  |\nabla^\sigma\widetilde\rho|_\sigma
  =\lambda\sin\rho\,|\nabla^\sigma\varphi|_\sigma
  \leq C\sqrt{W_{\max}}
  \leq Ce^{-c\tau},
\]
after changing $c$.  This proves
\eqref{eq:exponential-gradient-decay-spherical}.
\end{proof}
\section{Proof of Theorems}

\begin{proof}[Proof of Theorems~\ref{thm:main} and
\ref{thm:main-critical}]
The preceding estimates are uniform on every finite normalized time
interval.  In particular, Lemma~\ref{lem:normalized-C2-estimate-spherical}
shows that the principal curvature vector remains in a fixed compact
subset of the positive cone.  Hence the scalar equation
\eqref{eq:normalized-radial-flow-hat-h} is uniformly parabolic.  More
precisely, differentiating \eqref{eq:normalized-radial-flow-hat-h} in
$\tau$ shows that $\partial_\tau\widetilde\rho$ solves a linear
uniformly parabolic equation with bounded coefficients, so the
Krylov--Safonov theorem gives a uniform $C^\theta$ bound for
$\partial_\tau\widetilde\rho$, for some $\theta\in(0,1)$.  At each
fixed time, \eqref{eq:normalized-radial-flow-hat-h} can then be
written as
\[
  \sigma_k^{1/k}(\hat h[\widetilde\rho])
  =\left(
     \frac{\gamma_0\widetilde\rho-\partial_\tau\widetilde\rho}
          {\lambda^\beta f(\widetilde\rho/\lambda)\,v}
   \right)^{\!1/(k\alpha)},
\]
a uniformly elliptic equation whose left-hand side is concave in
$\nabla^2_\sigma\widetilde\rho$ and whose right-hand side is uniformly
$C^\theta$.  The Evans--Krylov theorem \cite{Krylov87} therefore gives
a uniform $C^{2,\theta}$ estimate.  The parabolic
Schauder estimates and the usual bootstrap argument then yield
\begin{equation}\label{eq:all-higher-order-estimates-spherical}
  \|\widetilde\rho(\cdot,\tau)\|_{C^m(\mathbb S^n)}
  \leq C_m,
  \qquad m=0,1,2,\ldots,
\end{equation}
uniformly in $\tau$.  These estimates also give long-time existence by
the standard continuation criterion.

By Lemma~\ref{lem:exponential-gradient-decay-spherical},
\begin{equation}\label{eq:first-derivative-exponential-decay-theorems}
  \|\nabla^\sigma\widetilde\rho(\cdot,\tau)\|_{C^0}
  \leq Ce^{-c\tau}.
\end{equation}
Combining this estimate with
\eqref{eq:all-higher-order-estimates-spherical} and applying the
interpolation inequalities on $\mathbb S^n$, we obtain, for every
$m\geq1$, constants $C_m,c_m>0$ such that
\begin{equation}\label{eq:all-spatial-derivatives-exponential-decay}
  \|(\nabla^\sigma)^m\widetilde\rho(\cdot,\tau)\|_{C^0}
  \leq C_me^{-c_m\tau}.
\end{equation}
In particular, the oscillation of $\widetilde\rho$ decays
exponentially.  Thus the only possible limit is spatially constant.

To identify the limit, let
\[
  \overline\rho(\tau)
  :=\frac{1}{|\mathbb S^n|}
    \int_{\mathbb S^n}\widetilde\rho(\theta,\tau)\,d\mu_\sigma.
\]
Using \eqref{eq:normalized-radial-flow-hat-h}, the expansion
$\lambda^{-1}\cot(\widetilde\rho/\lambda)
=\widetilde\rho^{-1}+O(\lambda^{-2})$, and
\eqref{eq:all-spatial-derivatives-exponential-decay}, we reduce the
evolution of the constant mode to the corresponding spherical ODE,
up to an exponentially decaying error.

In the supercritical case, put
$\mu:=\beta-k\alpha-1>0$.  The assumptions on $f$ give
\begin{equation}\label{eq:supercritical-average-asymptotic-ode}
  \overline\rho'
  =\gamma_0\overline\rho
     \bigl(1-\overline\rho^\mu\bigr)
   +O(e^{-c\tau}).
\end{equation}
The equilibrium $1$ of the limiting ODE is strictly stable.  The
two-sided $C^0$ estimate and the standard comparison argument for
\eqref{eq:supercritical-average-asymptotic-ode} therefore imply
\begin{equation}\label{eq:supercritical-average-convergence}
  |\overline\rho(\tau)-1|\leq Ce^{-c\tau}.
\end{equation}
Together with
\eqref{eq:all-spatial-derivatives-exponential-decay}, this proves that
$\widetilde\rho$ converges exponentially to $1$ in every $C^m$ norm.

In the critical case $\beta=1+k\alpha$, the leading terms in the
normalized equation cancel.  The condition
$g(r)=O(r^{1+k\alpha+\delta})$, the spherical expansion above, and
\eqref{eq:all-spatial-derivatives-exponential-decay} give
\begin{equation}\label{eq:critical-average-integrable-derivative}
  |\overline\rho'(\tau)|\leq Ce^{-c\tau}.
\end{equation}
Consequently, there exists $R_\infty>0$ such that
\begin{equation}\label{eq:critical-average-convergence}
  |\overline\rho(\tau)-R_\infty|
  \leq Ce^{-c\tau}.
\end{equation}
Again using
\eqref{eq:all-spatial-derivatives-exponential-decay}, we conclude that
$\widetilde\rho$ converges exponentially to $R_\infty$ in every
$C^m$ norm.

Thus, in either case, the normalized radial graphs converge smoothly
and exponentially to a constant radial graph, namely a geodesic
sphere centred at $o$.  Since
$\rho=\widetilde\rho/\lambda$ and $\lambda\to\infty$, the original
hypersurfaces contract smoothly to $o$ as $t\to\infty$.  This completes
the proofs of both theorems.
\end{proof}
%%%%%%%%%%%%%%%%%%%%%%%%%

%%%%%%%%%%%%%%%%%%%
{\bf AI usage.}
During the preparation of this work, the authors used GPT-5.6 Sol and KIMI K3 for exploratory computations, possible proof directions, and editorial polishing.  The authors have thoroughly checked all mathematical derivations and proofs and take full responsibility for the entire content of this manuscript.

%%%%%%%%%%%%%%%%%%%%%
{\bf Acknowledgements.}
W. Sheng was partially supported by National Key R$\&$D Program of China (No. 2022YFA1005500) and Natural Science Foundation of China under Grant No. 12571063.

%%%%%%%%%%%%%%%%%%%%%%%

\end{document}